\documentclass[10pt]{article}

\usepackage[colorlinks,linkcolor=blue,citecolor=blue]{hyperref}
\usepackage{tabulary} 
\usepackage{amsthm}
\usepackage{multirow}
\usepackage{marginnote}
\usepackage{a4wide}
\usepackage{amssymb}
\usepackage{amsfonts}
\usepackage{amsmath}
\usepackage{mathrsfs}
\usepackage{tikz}
\usetikzlibrary{arrows,matrix}
\usetikzlibrary{positioning}
\usepackage{mdframed} 
\usepackage{lipsum} 
\usepackage{extarrows} 
\usepackage{color}
\usepackage{enumitem} 
\usepackage{tocloft} 
\usepackage{titlesec} 

\usepackage[all]{xy} 

\input xy
\xyoption{arrow} \xyoption{matrix}

\usepackage{amsmath,amssymb} 
\usepackage{amsthm} 
\usepackage{mathtools} 
\usepackage{geometry} 
\usepackage{tikz}
\usepackage{tikz-cd}

\date{}

\newtheorem{proposition}{Proposition}[section]
\newtheorem{theorem}[proposition]{Theorem}
\newtheorem{lemma}[proposition]{Lemma}
\newtheorem{example}[proposition]{Example}
\newtheorem{definition}[proposition]{Definition}
\newtheorem{corollary}[proposition]{Corollary}

\def\der{\partial }

\def\nFM0{{\nu }_{F,M_0}}
\def\nFN0{{\nu }_{F,N_0}}
\def\nGN0{{\nu }_{G,N_0}}

\def\N0{ {\bf N}_0 }

\def\t{\otimes}
\def\g{\gamma}

\def\ra{\rightarrow}

\def\Xpm{X^{\pm }}

\def\s{\sigma}

\def\l1{{\lambda}_1}

\def\a{\alpha}
\def\a0{ {\alpha }_0}
\def\a1{ {\alpha }_1}

\def\l{\lambda}

\def\nFGM0{{\nu }_{F,G,M_0}}

\def\nFN0{{\nu}_{F,N_0}}

\def\sm{{\sigma}^m}

\def\sm1{{\sigma}^{-1}}

\def\smtp1{{\sigma}^{-t+1}}

\def\S1{S^{-1}}

\def\Xpm1{X^{\pm 1}_1}

\def\sPM1{{\sigma }^{\pm 1}}
\def\sMP1{{\sigma }^{\mp 1 }}

\def\b{\beta}
\def\d{\delta}

\def\di{{\rm d.ind}}

\def\L{\Lambda}

\def\Ytm1{Y^{t-1}}
\def\Yim1{Y^{i-1}}

\def\bK{\overline{K}}

\def\ker{ {\rm ker } }

\def\gcd{ {\rm gcd } }

\def\SL2Z{ {\rm SL}_2({\bf Z}) }

\def\th{ \theta }

\def\Gp1{ G^{1 , 1 } }
\def\P11{ P^{-1 , 1 } }
\def\Pp1{ P^{1 , 1 } }

\def\th{\theta}

\def\nCLsr{{}^\nu\kern-2pt {\cal L}^{\sigma , \rho  }}
\def\nP{{}^\nu \kern-2pt P}
\def\nL{{}^\nu\kern-2pt L}
\def\nLL{{}^\nu\kern-2pt \Lambda}
\def\nPsr{{}^\nu\kern-2pt P^{\sigma , \rho  }}
\def\nLsr{{}^\nu\kern-2pt L^{\sigma , \rho  }}
\def\nuCL{{}^\nu\kern-2pt  {\cal L}}
\def\nCLsr{{}^\nu\kern-2pt {\cal L}^{\sigma , \rho  }}
\def\nCL1m{{}^\nu\kern-2pt {\cal L}^{-1 , 1  }}
\def\x1nu{x^\frac{1}{\nu}}
\def\xm1nu{x^{-\frac{1}{\nu}}}

\def\ra{\rightarrow }

\def\CB{{\cal B}}

\def\nAM0{{\nu }_{{\cal A},M_0}}
\def\nAN0{{\nu }_{{\cal A},N_0}}

\def\SL{{\rm SL}}

\def\di!{\frac{\der^i}{i!}}
\def\dik!{\frac{\der^k_i}{k!}}

\def\gl{\mathfrak{l}}

\def\N{\mathbb{N}}

\def\0{\overline{0}}
\def\1{\overline{1}}

\def\Ln1{\L_{n,\overline{1}}}

\def\a1{a_{\overline{1}}}

\def\S{\Sigma}

\def\vn1{\overrightarrow{n-1}}

\def\gl{{\rm gl}}
\def\sl{{\rm sl}}

\def\F{\mathbb{F}}

\def\Inn{{\rm Inn}}

\def\mJ{\mathbb{J}}
\def\mI{\mathbb{I}}

\def\mF{\mathbb{F}}

\def\K1{{\rm K}_1}

\def\hmI1{\widehat{\mI_1}}
\def\tmI1{\widetilde{\mI_1}}
\def\tmJ1{\widetilde{\mJ_1}}
\def\hB1{\widehat{B_1}}
\def\hCB1{\widehat{\CB_1}}

\def \S{\mathcal{S}}

\def\CB{{\cal B}}

\def\hB{\hat{B}}

\def\sl2{\mathfrak{sl}_2}

\def\Irr{{\rm Irr}}

\def\sl2{\mathfrak{sl}_2}
\def\gl2{\mathfrak{gl}_2}

\def\res{{\rm res}}

\def\Irr{{\rm Irr}}
\def\b1{\overline{1}}

\def\res{{\rm res}}

\def\gl{{\mathfrak{l}}}

\makeatletter
\newenvironment{proof*}[1][\proofname]{\par
  \pushQED{\qed}%
  \normalfont \partopsep=\z@skip \topsep=\z@skip
  \trivlist
  \item[\hskip\labelsep
        \itshape
    #1\@addpunct{.}]\ignorespaces
}{%
  \popQED\endtrivlist\@endpefalse
}
\makeatother

\begin{document}

\author{V. V. \  Bavula 
}
\title{Explicit descriptions of the subfields 
$(NL)^{pi}$ and $(NL)^{pi}(NL)^{sep}$ of  $NL$ and new explicit  criteria for $NL = (NL)^{pi}(NL)^{sep}$}

\maketitle

\begin{abstract}
Let $L=K(\theta)\simeq K[x]/f(x)$ be a simple field extension  in prime characteristic $p>0$, $L^{sep}$ and $L^{pi}$ be the maximal separable and purely inseparable subfields of $L$, respectively. 
Let $N/K$ be a purely inseparable field extension. For the field extensions $L/K$ and $NL/N$, the aim of the paper is  to give explicit descriptions of the following subfields and their degrees in terms of the coefficients of the polynomial $f$ and two numerical field invariants $m_f$ and $m_{f,N}$: $L^{pi}$, $L^{pi}L^{sep}$, $(NL)^{pi}$ and  $(NL)^{pi}(NL)^{sep}$.  From these results, we derive new explicit criteria for $L=L^{pi}L^{sep}$ and $NL=(NL)^{pi}(NL)^{sep}$.\\




{ \small \tableofcontents}

\end{abstract}


\section{Introduction}\label{INTRO}

The following notation is fixed (unless it is stated otherwise): $K$ is a field of prime characteristic $p>0$, $\bK$ is the algebraic closure of $K$,  $K[x]$ is a polynomial $K$-algebra in a variable $x$, $\Irr_m(K[x])$ is the set of monic irreducible polynomials over the field $K$ and  $L/K$ is a finite  field extension. If, in addition,  
 $L/K$ is a simple finite  field extension then  $L=K(\th)=K[x]/(f(x))$ where $f(x)\in \Irr_m(K[x])$ and   
$$
f(x)=f^{sep}(x^{p^{n}})=\sum_{i=0}^s\l_ix^{ip^n}\in \Irr_m(K[x])\;\;  (\l_i\in K\;\; {\rm  and}\;\;  \l_s=1)
$$
 is its separable presentation (see (\ref{f=fsepxn})), $L^{pi}$  and $L^{sep}$ are maximal purely inseparable and separable subfields of the field extension $L/K$, respectively. For algebraic  field extensions $A/K$ and $B/K$, we denote  by $AB$ their compositum in $\bK$. There is a natural $K$-algebra epimorphism $A\t B\ra AB$, $a\t b\mapsto ab$ which is not an isomorphism, in general, where $\t = \t_K$. If $A/K$ is a purely inseparable field extension and $B/K$ is a separable field extension then $A\t B\simeq AB$. In general, for a  field extension $L/K$, the subfield $L^{pi}L^{sep}\simeq L^{pi}\t L^{sep}$ of $L$ is a {\em proper}  subfield.  All missing definitions in the paper  are standard  and can be found, say  in \cite{Lang-Algebra}.\\

{\bf Explicit descriptions of the subfields $L^{pi}$, $L^{sep}$ and $L^{pi}L^{sep}$ of a simple field extension $L/K$.} For a simple field extension $L/K$, Theorem \ref{27Apr26} gives explicit descriptions of the  subfields $L^{pi}/K$, $L^{sep}/K$ and $L^{pi}L^{sep}/K$ in terms of the coefficients of the polynomial $f$ and its inseparability degree. Notice that the description of $L^{sep}/K$ is a well-known result. 
 Theorem \ref{27Apr26} gives also explicit numerical values for the degrees  $[L^{pi}:K]$, $[L^{sep}:K]$, $[L^{pi}L^{sep}:L^{pi}]$ and   $[L^{pi}L^{sep}:L^{sep}]$  and $[L:L^{pi}\t L^{sep}]$ where  a numerical invariant $m_f$ (Definition \ref{DEF=mf}) plays a key role.  The number $m_f$ is defined via the coefficients of the polynomial $f$. It turns out that it is a field invariant for $L/K$, see 
 Theorem \ref{27Apr26}.(2).  It also reveals the reason why, in general,  the field $L^{pi}L^{sep}$ is a {\em proper} subfield of $L$.  Theorem \ref{27Apr26} yields several new   criteria for $L=L^{pi}L^{sep} (=L^{pi}\t L^{sep})$, see Theorem \ref{B27Apr26} and  Theorem \ref{A7May26}.\\
\[
\begin{tikzpicture}[scale=1.8, every node/.style={font=\normalsize}]

\node (L) at (0,2.6) {$L=K(\th)$};
\node (A) at (0,0) {$K$};
\node (B) at (-1.2,0.9) {$L^{pi}=K\bigg(\l_0^\frac{1}{p^{m_f}}, \ldots ,  \l_{s-1}^\frac{1}{p^{m_f}}\bigg)$};
\node (D) at (1.2,0.9) {$L^{sep}=K\Big(\th^{p^n}\Big)$};
\node (C) at (0,1.7) {$L^{pi}L^{sep} \;=\; L^{pi}\!\otimes L^{sep}=L^{pi}\Big(\th^{p^{n-m_f}}\Big)$};

\draw (L) -- node[right] {$p^{\,n-{m_f}}$} (C);

\draw (C) -- node[right] {$p^{{m_f}}$} (D);
\draw (D) -- node[right] {$\deg(f^{sep})$} (A);
\draw (A) -- node[left]  {$p^{{m_f}}$} (B);
\draw (B) -- node[left]  {$\deg(f^{sep})$} (C);

\end{tikzpicture}
\]

{\bf (Theorem \ref{27Apr26})} {\em  Suppose that $K$ is a field of prime characteristic $p>0$, $L/K$ is a simple finite  field extension and $L=K(\th)=K[x]/(f(x))$ where $f(x)=f^{sep}(x^{p^{n}})=\sum_{i=0}^s\l_ix^{ip^n}\in \Irr_m(K[x])$, $\l_i\in K$, $\l_s=1$ and $m:=m_f$. Then: 
\begin{enumerate}

\item $L^{sep}=K(\th^{p^n})\simeq K[x]/(f^{sep}(x))$, $[L^{sep}:K]=\deg(f^{sep}(x))=s$ and $f^{sep}(x)\in \Irr_m(K[x])$ is the minimal polynomial of the element $\th^n$ over the field $K$.

\item  $L^{pi}=K\Big(\l_0^\frac{1}{p^m}, \ldots ,  \l_{s-1}^\frac{1}{p^m}\Big)$,   $[L^{pi}:K]=p^m$ and $m=\max\Big\{m'=0,1,\ldots , n \, \Big| \, \th^{p^{n-m'}}\in L^{pi}L^{sep} \Big\}=\max\Big\{m'=0,1,\ldots , n \, \Big| \, L^{p^{n-m'}}\subseteq  L^{pi}L^{sep} \Big\}$. In particular, the number $m$ is an isomorphism invariant of the field extension $L/K$.

\item $L^{pi}L^{sep}=L^{pi}\t L^{sep}=L^{pi}(\th^{p^{n-m}})\simeq L^{pi}[x]/(f^{sep\frac{1}{p^m}})$, $[L^{pi}L^{sep}:L^{pi}]= s$,  $[L^{pi}L^{sep}:L^{sep}]= p^m$  
and the polynomial $f^{sep\frac{1}{p^m}}:=\sum_{i=0}^s\l_i^\frac{1}{p^m}x^i\in \Irr_m(L^{pi}[x])$ is the minimal polynomial of the element $\th^{p^{n-m}}$ over the field $L^{pi}$.

\item $L=L^{pi}\t L^{sep}(\th)=L^{pi}\t L^{sep}[x]/\Big(x^{p^{n-m}}- \th^{p^{n-m}}\Big)$,  $[L:L^{pi}\t L^{sep}]= p^{n-m}$ and $x^{p^{n-m}}- \th^{p^{n-m}}\in \Irr_m(L^{pi}\t L^{sep}[x])$ is the minimal polynomial of the element $\th$ over the field $L^{pi}\t L^{sep}$. The finite field extension $L/L^{pi}\t L^{sep}$ is a simple purely inseparable field extension of exponent $n-m$. 

\item $(L/L^{pi})^{sep}=L^{pi}L^{sep}/L^{pi}$.

\end{enumerate}
}

Lemma \ref{a17May26} provides two different methods for determining the invariant $m_f$.\\

{\bf Explicit descriptions of the subfields $(NL)^{pi}$ and $(NL)^{pi}(NL)^{sep}$  of $NL/N$ where $L/K$ is a simple field extension and $N/K$ is a purely inseparable field extension.}
For a simple field extension $L/K=K(\th)/K$ and a purely inseparable field extension $N/K$ (not necessarily finite),
 Theorem \ref{A6May26} describes the structure of the compositum $NL$, the degree $[NL:N]$ and the minimal polynomial for the simple field extension $NL/N$. \\

{\bf (Theorem \ref{A6May26})}  
{\em Suppose that $K$ is a field of prime characteristic $p>0$, $L/K$ is a simple finite  field extension and $L=K(\th)=K[x]/(f(x))$ where $f(x)=f^{sep}(x^{p^{n}})=\sum_{i=0}^s\l_ix^{ip^n}\in \Irr_m(K[x])$, $\l_i\in K$, $\l_s=1$ and  $N/K$ is a purely inseparable field extension. Then: 
\begin{enumerate}

\item $f_N=\sum_{i=0}^s\l_i^\frac{1}{p^{m_{f,N}}}x^{ip^{n-m_{f,N}}}\in \Irr_m(N[x])$, $f_N\in \Irr_m (M_{f,N}[x])$, $f= f_N^{p^{m_{f,N}}}$,   $\deg (f_N)=sp^{n-m_{f,N}}$ where $M_{f,N}=K\bigg(\l_0^\frac{1}{p^{m_{f,N}}}, \ldots ,  \l_{s-1}^\frac{1}{p^{m_{f,N}}}\bigg)$.

\item $N(\th)\simeq N[x]/(f_N)$ and  $[N(\th):N]=\deg (f_N)=sp^{n-m_{f,N}}$.


\end{enumerate}
}

 Theorem \ref{VB-27Apr26} yields explicit descriptions of the following subfields of  $NL/N$ --  $(NL)^{pi}/N$, $(NL)^{sep}/N$ and their  compositum $(NL)^{pi}(NL)^{sep}/N$ --  in terms of the coefficients of the minimal  polynomial $f$ of the element  $\th\in L$ over $K$ and two natural numbers (that are field invariants) $m_{f,N}$ and $m_{f,N(\th)}$ associated with  $f$,  $N$ and $NL$ (Definition \ref{Def=mfN} and Definition \ref{DEF=mfNth}). We compute explicit numerical values for the following degrees:  $[(NL)^{pi}:N]$, $[(NL)^{sep}:N]$, $[(NL)^{pi}(NL)^{sep}:(NL)^{pi}]$,  $[(NL)^{pi}(NL)^{sep}:(NL)^{sep}]$  and $[L:(NL)^{pi}(NL)^{sep}]$.   \\


{\bf (Theorem \ref{VB-27Apr26})} 
{\em Suppose that $K$ is a field of prime characteristic $p>0$, $L/K$ is a simple finite  field extension and $L=K(\th)=K[x]/(f(x))$ where $f(x)=f^{sep}(x^{p^{n}})=\sum_{i=0}^s\l_ix^{ip^n}\in \Irr_m(K[x])$, $\l_i\in K$, $\l_s=1$   and $N/K$ is a purely inseparable field extension. Then (below $(NL)^{sep}:=(NL/N)^{sep}$ and $(NL)^{pi}:=(NL/N)^{pi}$): 
\begin{enumerate}

\item $(NL/N)^{sep}=NL^{sep}=N(\th^{p^{n-m_{f,N}}})\simeq N[x]/(f^{sep}_N)$, $[NL^{sep}:N]=\deg(f^{sep}_N(x))=s$ and $f^{sep}_N(x)\in \Irr_m(N[x])$ is the minimal polynomial of the element $\th^{n-m_{f,N}}$ over the field $N$.

\item  $(NL/N)^{pi}=N\bigg( 
\l_0^\frac{1}{p^{m_{f,N}+m_{f,N(\th)}}}, \ldots ,  \l_{s-1}^\frac{1}{p^{m_{f,N}+m_{f,N(\th)}}}\bigg)\supseteq NL^{pi}=N\bigg( 
\l_0^\frac{1}{p^{m_{f,N}}}, \ldots ,  \l_{s-1}^\frac{1}{p^{m_{f,N}}}\bigg)$ and
\begin{eqnarray*}
[(NL)^{pi}:N]&=& p^{m_{f,N(\th)}},\\
 m_{f,N(\th)}&=&\max\Big\{m'=0,1,\ldots , n-m_{f,N} \, | \, \th^{p^{n-m_{f,N}-m'}}\in (NL)^{pi}(NL)^{sep} \Big\}\\
 &=&\max\Big\{m'=0,1,\ldots , n-m_{f,N}  \, | \, (NL)^{p^{n-m_{f,N} -m'}}\subseteq  (NL)^{pi}(NL)^{sep} \Big\}.
\end{eqnarray*}
  In particular, the number $m_{f,N(\th)}$ is an isomorphism invariant of the field extension $NL/N$.

\item $(NL)^{pi}(NL)^{sep}=(NL)^{pi}\t_N (NL)^{sep}=(NL)^{pi}\Big(\th^{p^{n-m_{f,N}-m_{f,N(\th)}}}\Big)\simeq (NL)^{pi}[x]/(f^{sep}_{NL})$, 
\begin{eqnarray*}
[(NL)^{pi}(NL)^{sep}:(NL)^{pi}] &=& s,\;\;  [(NL)^{pi}(NL)^{sep}:(NL)^{sep}] = p^{m_{f,N(\th)}},\\
f^{sep}_{NL} &=& \sum_{i=0}^s\l_i^\frac{1}{p^{m_{f,N}+m_{f, N(\th)}}}x^i\in \Irr_m((NL/N)^{pi}[x])
\end{eqnarray*}
is the minimal polynomial of the element $\th^{p^{n-m_{f,N}-m_{f,N(\th)}}}$ over the field $(NL)^{pi}$. 

\item $NL=(NL)^{pi}\t_N (NL)^{sep}(\th)=(NL)^{pi}\t_N (NL)^{sep}[x]\bigg/\bigg(x^{p^{n-m_{f,N}-m_{f,N(\th)}}}- \th^{p^{n-m_{f,N}-m_{f,N(\th)}}}\bigg)$,  
\begin{eqnarray*}
[NL:(NL)^{pi}\t_N (NL)^{sep}]&=& p^{n-m_{f,N}-m_{f,N(\th)}},\\
 x^{p^{n-m_{f,N}-m_{f,N(\th)}}}- \th^{p^{n-m_{f,N}-m_{f,N(\th)}}}&\in & \Irr_m\Big((NL)^{pi}\t_N (NL)^{sep}[x]\Big)
\end{eqnarray*}
is the minimal polynomial of the element $\th$ over the field $(NL)^{pi}\t_N (NL)^{sep}$. The finite field extension $L/L^{pi}\t_N L^{sep}$ is a simple purely inseparable field extension of exponent $n-m_{f,N}-m_{f,N(\th)}$. 

\item $\Big(NL/(NL)^{pi}\Big)^{sep}=(NL)^{pi}(NL)^{sep}/(NL)^{pi}$.

\end{enumerate}
}

 Corollary  \ref{Na7May26} is  an explicit criterion for $(NL)^{pi}=N$. Theorem \ref{NB27Apr26} is an explicit criteria for $NL=(NL)^{pi}(NL)^{sep}$.\\

{\bf New Criteria for $L=L^{pi}L^{sep}$.}
 In the literature, there are several  criteria for $L=L^{pi}L^{sep}$, see Theorem \ref{C27Apr26} for detail:

\begin{itemize}

\item \textbf{(The Degree Criterion)} $[L:K] = [L^{pi}:K] \cdot [L^{sep}:K]$. 

\item \textbf{(Separability over the Purely Inseparable Part)} {\em The extension $L/L^{pi}$ is a separable field  extension.}
 
\item   \textbf{(Equality of the Inseparable Degree)}  $ [L^{pi}:K] = [L:K]_i $   {\em where $[L:K]_i$ denotes the inseparable degree of} $L/K$.
\end{itemize}

  Each finite field extension $L/K$  is the compositum $L=L_1\cdots L_\nu$ of simple finite field extensions $L_i=K(\th_i)\simeq K[x]/(f_i)$  where $f_i(x)=f_i^{sep}(x^{p^{n_i}})=\sum_{j=0}^{s_i}\l_{ij}x^{jp^{n_i}}\in \Irr_m(K[x])$ is the minimal polynomial of the element $\th_i$ over $K$ and $\deg (f_i)=s_ip^{n_i}$. Theorem \ref{A7May26} is a new explicit criterion  for $L = L^{pi}L^{sep}$ which is given in terms of the coefficients $\l_{ij}$ and the numbers $s_i$ and $n_i$.\\


{\bf (Theorem  \ref{A7May26})}  
{\em Suppose that $L/K$ is a finite field extension of prime characteristic $p>0$ which  is the compositum $L=L_1\cdots L_\nu$ of       simple field extensions $L_i=K(\th_i)\simeq K[x]/(f_i)$, $i=1, \ldots , \nu$ where $f_i(x)=f_i^{sep}(x^{p^{n_i}})=\sum_{j=0}^{s_i}\l_{ij}x^{jp^{n_i}}\in \Irr_m(K[x])$ and $\deg (f_i)=s_ip^{n_i}$.  Then the following statements are equivalent:

\begin{enumerate}

\item $L = L^{pi}L^{sep}$.

\item $\l_{ij}^\frac{1}{p^{n_i}}\in L^{pi}$ for $i=1, \ldots , \nu$ and $j=0,1, \ldots , s_i-1$. 

\item $L^{pi}= K\bigg(\l_{ij}^\frac{1}{p^{n_i}}\bigg| i=1, \ldots , \nu; j=0,1, \ldots , s_i-1\bigg)$.
     
\item $L^{pi}\supseteq K\bigg(\l_{ij}^\frac{1}{p^{n_i}}\bigg| i=1, \ldots , \nu; j=0,1, \ldots , s_i-1\bigg)$.

\end{enumerate}
}


\section{Explicit descriptions of the subfields $L^{pi}$ and $L^{pi}L^{sep}$  of a simple finite field extension $L/K$}\label{L=LpiLsep}

In this section, for a simple field extension $L/K$, Theorem \ref{27Apr26} yields explicit descriptions of the maximal purely inseparable and separable  subfields of $L/K$,  $L^{pi}/K$ and $L^{sep}/K$, and their  compositum $L^{pi}L^{sep}/K$ in terms of the coefficients of the polynomial $f$ and its inseparability degree. We compute explicit numerical values for the following degrees:  $[L^{pi}:K]$, $[L^{sep}:K]$, $[L^{pi}L^{sep}:L^{pi}]$,  $[L^{pi}L^{sep}:L^{sep}]$  and $[L:L^{pi}\t L^{sep}]$. 
 It further explains why the compositum $L^{pi}L^{sep}$ is typically strictly contained in 
$L$. For a simple field extension $L/K$, Corollary  \ref{a7May26} is  an explicit criterion for $L^{pi}=K$.\\

{\bf The equality $L^{pi} L^{sep}=L^{pi}\t L^{sep}$.} The equality $L^{pi} L^{sep}=L^{pi}\t L^{sep}$ is known result. 
We give an alternative, Galois-theoretic proof of this fact. We use this fact often in the paper.

\begin{theorem}\label{A27Apr26}
Suppose that $K$ is a field of prime characteristic $p>0$ and  $L/K$ is a field extension. Then  $L^{pi} L^{sep}=L^{pi}\t L^{sep}$.
\end{theorem}

\begin{proof} Clearly, $L^{pi}\cap L^{sep}=K$ and there is a $K$-epimorphism  $\pi : L^{pi}\t L^{sep}\ra L^{pi} L^{sep}$, $a\t b\mapsto ab$. We haver to show that $\ker (\pi) = \{0\}$.  Suppose that $\ker (\pi)  \neq  \{0\}$.  We seek a contradiction.  Then there is a nonzero element $\alpha=\sum_{i=1}^na_i\t b_i \in \ker (\pi)$ where $a_i\in  L^{pi}$ and $b_i\in  L^{sep}$. Let $A=K(a_1, \ldots , a_n)$ and $B=K(b_1, \ldots , b_n)$. Then    $\alpha \in A\t B$, the extension $A/K$ is a finite purely inseparable finite field extension and  the extension $B/K$ is  a separable finite field extension. Since $A\t B\subseteq  L^{pi}\t L^{sep}$, we may assume that 
$$
A =L^{pi}\;\; {\rm  and}\;\; B=L^{sep}.
$$
 Let $L^{nor}$ be the normal closure of $L^{sep}$ in $\bK$. Similarly, since $ L^{pi}\t L^{sep}\subseteq  L^{pi}\t L^{nor}$, we may assume that $L^{sep}=L^{nor}$, i.e. the finite field extension $L^{sep}$ is a Galois field  extension with Galois group $G(L^{sep}/K)$.   By the Primitive Element Theorem, $L^{sep}=K(\th )$ is a simple field extension where $\th \in L^{sep}$.  Recall that if $f(x)\in K[x]$ is the minimal polynomial of the element $\th$ over $K$ then  $$
 f(x)=\prod_{g\in G(L/K)}(x-g(\th))
 $$ 
 and $g(\th)\in L^{sep}$ for all $g\in G(L^{sep}/K)$. So, every automorphism $h\in G(L^{sep}/K)$ permutes the roots $\{ g(\th)\, | \, g\in G(L^{sep}/K)\}$ of the polynomial $f(x)$. 
  Since the field extension $L^{sep}/K$ is Galois, the field extension $L^{pi}L^{sep}/L^{pi}=L^{pi}(\th)/L^{pi}$ is also a Galois finite  field extension such that the restriction map  
  $$ \res: G(L^{pi}L^{sep}/L^{pi})\ra G(L^{sep}/K), \;\; \s\mapsto \s|_{L^{sep}
}  $$
  is a bijection (every automorphism of $L^{sep}/K$ is necessarily uniquely extended to an automorphism of $L^{pi}L^{sep}/L^{pi}$ by trivial action on the elements of $L^{pi}$). In particular,
  $$ 
  [L^{pi}L^{sep}:L^{pi}]=|G(L^{pi}L^{sep}/L^{pi})|=|G(L^{sep}/K)|=[L^{sep}:K].
  $$ 
It follows from $K\subseteq L^{pi}\subseteq L^{pi}L^{sep}=L^{pi}(\th)$ that 
$$
[L^{pi}L^{sep}:K]=[L^{pi}L^{sep}:L^{pi}][L^{pi}:K]=[L^{sep}:K][L^{pi}:K]=[L^{pi}\t L^{sep}:K],
$$
and so $L^{pi} L^{sep}=L^{pi}\t L^{sep}$.
\end{proof}

{\bf Explicit descriptions of the subfields $L^{pi}$, $L^{sep}$ and $L^{pi}L^{sep}$ of a simple field extension $L/K$.}

\begin{definition}
Suppose that $K$ is a field of prime characteristic $p>0$. Then each non-scalar polynomial $f(x)\in K[x]$ admits a unique presentation 

\begin{equation}\label{f=fsepxn}
 f(x)=f^{sep}(x^{p^n})\;\; {\rm where}\;\; f^{sep}(x)\in K[x]\;\; {\rm is\; a\; separable\; polynomial\; and}\;\; n\geq 0.
\end{equation}
The equality (\ref{f=fsepxn}) is called a {\bf separable presentation} of the polynomial $f(x)$. The polynomial $f^{sep}(x)$ is called the {\bf separable part} of $f$ and the natural number $n$ is called the {\bf inseparability degree} of $f(x)$ and denoted by $\deg_{\rm ins}(f)$. 
\end{definition}

For the polynomial $f(x)=\sum_{i\geq 0} \mu_ix^i$,  ${\rm coef}(f):= \{\mu_i\, | \, i\geq 0\}$ is the 
  set of its coefficients. Clearly, 
\begin{equation}\label{f=fsepxn-2}
 {\rm coef}(f)={\rm coef}(f^{sep}).
\end{equation}
Notice that 
\begin{equation}\label{f=fsepxn-1}
\deg(f)=p^n\deg (f^{sep})\;\; {\rm where}\;\; n=\deg_{\rm ins}(f).
\end{equation}
For the polynomial $f(x)$ as in (\ref{f=fsepxn}),   $f(x)=\sum_{i\geq 0} \l_ix^{ip^n}=\bigg(\sum_{i\geq 0} \l_i^\frac{1}{p^n} x^i\bigg)^{p^n}$, where $n=\deg_{\rm ins}(f)$, and so 

\begin{equation}\label{f=fsepxn-3}
f(x)=\bigg( {f^{sep}}^\frac{1}{p^n}\bigg)^{p^n}\;\; {\rm where}\;\;  {f^{sep}}^\frac{1}{p^n}:=\sum_{i\geq 0} \l_i^\frac{1}{p^n} x^i\in K\Big({\rm coef}(f)^\frac{1}{p^n}\Big)[x]
\end{equation}
is a {\em separable} polynomial over the purely inseparable finite  field extension $K\Big({\rm coef}(f)^\frac{1}{p^n}\Big)/K$. Clearly,
\begin{equation}\label{f=fsepxn-4}
{\rm roots}(f)={\rm roots}( {f^{sep}}^\frac{1}{p^n}).
\end{equation}

Suppose that $L=K(\th)=K[x]/(f)$ is a simple field extension, where $\th \in L$,  and the polynomial  $f(x)=f^{sep}(x^{p^n})\in K[x]$ is the minimal polynomial of $\th$. The following concepts are fundamental to providing explicit descriptions of the fields 
$L^{pi}$ and $L^{sep}$.

\begin{definition}\label{DEF=mf}
\begin{eqnarray*}
m_f:=m_{f,L}:=m_{f,L/K}&:=&\max\bigg\{ m'=0,1,\ldots , n\, \bigg| \, \l_i^\frac{1}{p^{m'}}\in L\;\; \text{ for all}\;\;i=0, \ldots , s-1\bigg\}\\
&=&\max\bigg\{ m'=0,1,\ldots , n\, \bigg| \, \l_i^\frac{1}{p^{m'}}\in L^{pi}\;\; \text{ for all}\;\;i=0, \ldots , s-1\bigg\},\\
f^{sep\frac{1}{p^{m_f}}}(x)&:=&\sum_{i=0}^s\l_i^\frac{1}{p^{m_f}}x^i\in L^{pi}[x].
\end{eqnarray*}
\end{definition}
Theorem \ref{27Apr26}.(2) shows that the number $m_f$ is an isomorphism invariant of the field extension $L/K$. 
 There is a field diagram where the edges are labelled by the degrees of the corresponding field extensions, see   Theorem \ref{A27Apr26} and  Theorem \ref{27Apr26} for details: 




\begin{equation}\label{L=LLL}
\begin{tikzpicture}[scale=1.8, every node/.style={font=\normalsize}]

\node (L) at (0,2.6) {$L=K(\th)$};
\node (A) at (0,0) {$K$};
\node (B) at (-1.2,0.9) {$L^{pi}=K\bigg(\l_0^\frac{1}{p^{m_f}}, \ldots ,  \l_{s-1}^\frac{1}{p^{m_f}}\bigg)$};
\node (D) at (1.2,0.9) {$L^{sep}=K\Big(\th^{p^n}\Big)$};
\node (C) at (0,1.7) {$L^{pi}L^{sep} \;=\; L^{pi}\!\otimes L^{sep}=L^{pi}\Big(\th^{p^{n-m_f}}\Big)$};

\draw (L) -- node[right] {$p^{\,n-{m_f}}$} (C);

\draw (C) -- node[right] {$p^{{m_f}}$} (D);
\draw (D) -- node[right] {$\deg(f^{sep})$} (A);
\draw (A) -- node[left]  {$p^{{m_f}}$} (B);
\draw (B) -- node[left]  {$\deg(f^{sep})$} (C);

\end{tikzpicture}
\end{equation}

For a simple field extension $L/K$, Theorem \ref{27Apr26} gives explicit descriptions of the  subfields $L^{pi}/K$, $L^{sep}/K$ and $L^{pi}L^{sep}/K$ in terms of the coefficients of the polynomial $f$ and its inseparability degree. Theorem \ref{27Apr26} gives also explicit numerical values for the degrees  $[L^{pi}:K]$, $[L^{sep}:K]$, $[L^{pi}L^{sep}:L^{pi}]$ and   $[L^{pi}L^{sep}:L^{sep}]$  and $[L:L^{pi}\t L^{sep}]$. It also reveals the reason why, in general,  the field $L^{pi}L^{sep}$ is a {\em proper} subfield of $L$.  Theorem \ref{27Apr26} yields several new   criteria for $L=L^{pi}L^{sep} (=L^{pi}\t L^{sep})$, see Theorem \ref{B27Apr26} and  Theorem \ref{A7May26}.


\begin{theorem}\label{27Apr26}
Suppose that $K$ is a field of prime characteristic $p>0$, $L/K$ is a simple finite  field extension and $L=K(\th)=K[x]/(f(x))$ where $f(x)=f^{sep}(x^{p^{n}})=\sum_{i=0}^s\l_ix^{ip^n}\in \Irr_m(K[x])$, $\l_i\in K$, $\l_s=1$ and $m:=m_f$. Then: 
\begin{enumerate}

\item $L^{sep}=K(\th^{p^n})\simeq K[x]/(f^{sep}(x))$, $[L^{sep}:K]=\deg(f^{sep}(x))=s$ and $f^{sep}(x)\in \Irr_m(K[x])$ is the minimal polynomial of the element $\th^n$ over the field $K$.

\item  $L^{pi}=K\Big(\l_0^\frac{1}{p^m}, \ldots ,  \l_{s-1}^\frac{1}{p^m}\Big)$,   $[L^{pi}:K]=p^m$ and $m=\max\Big\{m'=0,1,\ldots , n \, \Big| \, \th^{p^{n-m'}}\in L^{pi}L^{sep} \Big\}=\max\Big\{m'=0,1,\ldots , n \, \Big| \, L^{p^{n-m'}}\subseteq  L^{pi}L^{sep} \Big\}$. In particular, the number $m$ is an isomorphism invariant of the field extension $L/K$.

\item $L^{pi}L^{sep}=L^{pi}\t L^{sep}=L^{pi}(\th^{p^{n-m}})\simeq L^{pi}[x]/(f^{sep\frac{1}{p^m}})$, $[L^{pi}L^{sep}:L^{pi}]= s$,  $[L^{pi}L^{sep}:L^{sep}]= p^m$  
and the polynomial $f^{sep\frac{1}{p^m}}:=\sum_{i=0}^s\l_i^\frac{1}{p^m}x^i\in \Irr_m(L^{pi}[x])$ is the minimal polynomial of the element $\th^{p^{n-m}}$ over the field $L^{pi}$.

\item $L=L^{pi}\t L^{sep}(\th)=L^{pi}\t L^{sep}[x]/\Big(x^{p^{n-m}}- \th^{p^{n-m}}\Big)$,  $[L:L^{pi}\t L^{sep}]= p^{n-m}$ and $x^{p^{n-m}}- \th^{p^{n-m}}\in \Irr_m(L^{pi}\t L^{sep}[x])$ is the minimal polynomial of the element $\th$ over the field $L^{pi}\t L^{sep}$. The finite field extension $L/L^{pi}\t L^{sep}$ is a simple purely inseparable field extension of exponent $n-m$. 

\item $(L/L^{pi})^{sep}=L^{pi}L^{sep}/L^{pi}$.

\end{enumerate}
\end{theorem}

\begin{proof}  1. By the definition, the  polynomial $f^{sep}\in K[x]$ is a separable polynomial such that $f^{sep}(\th^{p^n})=f(\th)=0$. Therefore, $K(\th^{p^n})\subseteq L^{sep}$. In particular, the field extension $L^{sep}/K(\th^{p^n})$ is a separable field extension.  In fact, the equality holds, 
$$
K(\th^{p^n})= L^{sep}.
$$
 This follows from the field inclusions $K\subseteq  K(\th^{p^n})\subseteq L^{sep}\subseteq L=K(\th)$ and the facts that field extension $L/K(\th^{p^n})=K(\th )/K(\th^{p^n})$ is a purely inseparable field extension and its subfield extension $L^{sep}/K(\th^{p^n})$ is a separable field extension. 

Since $f(x)=f^{sep}(x^{p^n})\in \Irr_m(K[x])$, we have that $f^{sep}(x)\in \Irr_m(K[x])$. Now, 
$$
L^{sep}=K(\th^{p^n})\simeq K[x]/(f^{sep}(x)). 
$$ 
Hence, $[L^{sep}:K]=\deg(f^{sep}(x))=s$ and $f^{sep}(x)\in \Irr_m(K[x])$ is the minimal polynomial of the element $\th^{p^n}$ over the field $K$.

2--4.  By statement 1,   $\th^{p^n}\in L^{sep}$. This implies that the field extension $L/L^{pi}L^{sep}=K(\th)/L^{pi}L^{sep}$ is a purely inseparable field extension. Now, the equality  
$$
m^*:=\max\Big\{m'\in \N \, | \, \th^{p^{n-m'}}\in L^{pi}L^{sep} \Big\}=\max\Big\{m'\in \N \, | \, L^{p^{n-m'}}\subseteq  L^{pi}L^{sep} \Big\}
$$ 
follows from the equality $L=K(\th)$. By the definition of the number $m^*$ and the equality $L^{sep}=K(\th^{p^n})$, 
$$
L^{pi}L^{sep}=L^{pi}(\th^{p^{n-m^*}}).
$$

(i) {\em $L=L^{pi}\t L^{sep}(\th)=L^{pi}\t L^{sep}[x]/\Big(x^{p^{n-m^*}}- \th^{p^{n-m^*}}\Big)$,  $[L:L^{pi}\t L^{sep}]= p^{n-m^*}$    and $x^{p^{n-m^*}}- \th^{p^{n-m^*}}\in \Irr_m(L^{pi}\t L^{sep}[x])$ is the minimal polynomial of the element $\th$ over the field $L^{pi}\t L^{sep}$. The finite field extension $L/L^{pi}\t L^{sep}$ is a simple purely inseparable field extension of exponent} $n-m^*$: The statement (i) follows from the definition of the number $m^*$.

(ii) $[L:L^{sep}]=p^n$ {\em and}  $[L^{pi}L^{sep}:L^{sep}]=p^{m^*}$: By statement 1, $[L^{sep}:K]=\deg (f^{sep}(x))=s$, and so 
$$
[L:L^{sep}]=\frac{[L:K]}{[L^{sep}:K]}=\frac{\deg (f)}{s}=\frac{sp^n}{s}=p^n.
$$
Now, the second  equality in the statement (ii) follows from the statement (i),
$$
[L^{pi}L^{sep}:L^{sep}]=\frac{[L:L^{sep}]}{[L:L^{pi}L^{sep}]}=\frac{p^n}{p^{n-m^*}}=p^{m^*}.
$$

(iii) $[L^{pi}L^{sep}:K]=sp^{m*}$ {\em and} $[L^{pi}:K]=p^{m^*}$: 
$$
[L^{pi}L^{sep}:K]= [L^{pi}L^{sep}:L^{sep}][L^{sep}:K]=p^{m*}s,
$$

$$
[L^{pi}:K] = \frac{[L^{pi}:K][L^{sep}:K]}{[L^{sep}:K]} 
 =  \frac{[L^{pi}\t L^{sep}:K]}{[L^{sep}:K]}=\frac{p^{m^*}s}{s}=p^{m^*}.
$$

(iv) $[L^{pi}L^{sep}:L^{pi}]=s$:

$$
[L^{pi}L^{sep}:L^{pi}]= \frac{[L^{pi}L^{sep}:K]}{[L^{pi}:K]}=\frac{sp^{m^*}}{p^{m^*}}  =s.
$$

Notice that $f^{sep,\frac{1}{p^{m^*}}}(x):=\sum_{i=0}^s\l_i^\frac{1}{p^{m^*}}x^i\in \bK^{pi}[x]$ where $\bK^{pi}/K$ is the maximal purely inseparable field extension in $\bK/K$ ($\bK$ is the algebraic closure of $K$). Then  the polynomial $f^{sep,\frac{1}{p^{m^*}}}(x)$ is a monic polynomial of degree $\deg\Big(f^{sep,\frac{1}{p^{m^*}}}(x)\Big)=s$ and the element $\th^{p^{n-m^*}}$ is a root of it:
$$
f^{sep,\frac{1}{p^{m^*}}}(\th^{p^{n-m^*}})=
\sum_{i=0}^s\l_i^\frac{1}{p^{m^*}}\th^{ip^{n-m^*}}=\bigg(\sum_{i=0}^s\l_i \th^{ip^n} \bigg)^\frac{1}{p^{m^*}}=\Big( f(\th)\Big)^\frac{1}{p^{m^*}}= 0^\frac{1}{p^{m^*}}  =0.
$$
 
By the statement (iv) and the equality $ L^{pi}(\th^{p^{n-m^*}})=L^{pi}L^{sep}$, 
$$
[L^{pi}(\th^{p^{n-m^*}}):L^{pi}]=[L^{pi}L^{sep}:L^{pi}]=s=\deg(f^{sep,\frac{1}{p^{m^*}}}).
$$
 This implies that the polynomial $f^{sep,\frac{1}{p^{m^*}}}(x)=\sum_{i=0}^s\l_i^\frac{1}{p^{m^*}}x^i\in L^{pi}[x]$ is the minimal polynomial of the element $\th^{p^{n-m^*}}\in L^{pi}L^{sep}$ over the field $L^{pi}$. In particular, all its coefficients belong to the field $L^{pi}$, i.e. 
 $$\l_i^\frac{1}{p^{m^*}}\in L^{pi}\;\; \text{ for all}\;\;i=1, \ldots , s-1.
 $$
  Therefore, $f^{sep,\frac{1}{p^{m^*}}}(x)\in \Irr_m(M^*[x])$ where $M^*:=K\bigg(\l_0^\frac{1}{p^{m^*}}, \ldots ,  \l_{s-1}^\frac{1}{p^{m^*}}\bigg)\subseteq L^{pi}$.

The polynomial $f^{sep,\frac{1}{p^{m^*}}}(x)=\sum_{i=0}^s\l_i^\frac{1}{p^{m^*}}x^i\in \in \Irr_m(M^*[x])$ is a separable polynomial over the field $M^*$ (since it's derivative is a nonzero polynomial). So, we have proven the statement (v).

(v) $f^{sep,\frac{1}{p^{m^*}}}(x)\in \Irr_m(M^*[x])$ {\em and the field extension  $M^*(\th^{p^{n-m^*}})/M^*$ is a separable field extension of degree} $[M^*(\th^{p^{n-m^*}}):M^*]=\deg\Big(f^{sep,\frac{1}{p^{m^*}}}(x)\Big)=s$. 

(vi) $L^{pi}(\th^{p^{n-m^*}})=M^*(\th^{p^{n-m^*}})$: There is a chain of fields $$
L^{sep}=K(\th^{p^n})\subseteq M^*(\th^{p^{n-m^*}})\subseteq L^{pi}(\th^{p^{n-m^*}})=L^{pi}L^{sep}\subseteq L=K(\th).$$
Since $p^{n-m^*}\stackrel{{\rm (i)}}{=}[L:L^{pi}L^{sep}]=[L:L^{pi}(\th^{p^{n-m^*}})]\leq  [L:M^*(\th^{p^{n-m^*}})]\leq p^{n-m^*}$ , we must have 
$$
[L:L^{pi}(\th^{p^{n-m^*}})]= [L:M^*(\th^{p^{n-m^*}})].
$$
Now, the inclusion  $L^{pi}(\th^{p^{n-m^*}})\supseteq M^*(\th^{p^{n-m^*}})$ implies the equality  
$L^{pi}(\th^{p^{n-m^*}})=M^*(\th^{p^{n-m^*}})$.

(vii) $[M^*:K]=p^{m^*}$: By the statement (vi), there is a diagram of fields where the numbers at the edges are the degrees of the corresponding field extensions: 

\[
\begin{tikzpicture}[scale=1.8, every node/.style={font=\normalsize}]

\node (A) at (0,0) {$K$};
\node (B) at (-1.2,0.9) {$M^*$};
\node (D) at (1.2,0.9) {$L^{sep}=K(\th^{p^n})$};
\node (C) at (0,1.7) {$L^{pi}(\th^{p^{n-m^*}})=M^*(\th^{p^{n-m^*}})$};


\draw (C) -- node[right] {$p^{m^*}$} (D);
\draw (D) -- node[right] {$s$} (A);
\draw (A) -- node[left]  {$$} (B);
\draw (B) -- node[left]  {$s$} (C);

\end{tikzpicture}
\]

The equalities $[L^{sep}:K]=s$,   $[M^*(\th^{p^{n-m}}):M^*]=s$ and 
$$
[M^*(\th^{p^{n-m}}):L^{sep}]= [L^{pi}(\th^{p^{n-m^*}}):L^{sep}]= [L^{pi}L^{sep}:L^{sep}]    =p^{m^*}
$$
 follow from  statement 1 and  the statements (v) and  (ii), respectively. Now,  the diagram yields the equality $[M^*:K]= p^{m^*}$:
 $$
 [M^*:K]=\frac{[M^*(\th^{p^{n-m}}):L^{sep}][L^{sep}:K]}{[M^*(\th^{p^{n-m}}):M^*]}= \frac{p^{m^*}s}{s}
  =p^{m^*}.
  $$ 

(viii) $L^{pi}=M^*$ {\em and} $\Big( L^{pi}\Big)^{p^{m^*}}\subseteq K$: By the statements (iii) and  (vii), $[L^{pi}:K]=p^{m^*}=[M^*:K]$. Then the equality 
 $L^{pi}=M^*$ follows from the inclusion $L^{pi}\supseteq M^*$. The equality  $L^{pi}=M^*$ yields the inclusion 
 $$
 \Big( L^{pi}\Big)^{p^{m^*}}=\Big( M^*\Big)^{p^{m^*}}=\bigg( K\bigg(\l_0^\frac{1}{p^{m^*}}, \ldots ,  \l_{s-1}^\frac{1}{p^{m^*}}\bigg)\bigg)^{p^{m^*}}\subseteq K.
 $$
(ix) $(L/L^{pi})^{sep}=L^{pi}L^{sep}/L^{pi}$:
 Since the field extension $L^{pi}L^{sep}/L^{pi}$ is a separable field extension, we have the inclusion 
 $$(L/L^{pi})^{sep}\supseteq L^{pi}L^{sep}/L^{pi}.
 $$ 
Since the field extension $L/L^{pi}L^{sep}=K(\th)/L^{pi}(\th^{p^{n-m^*}})$ is purely inseparable, we must have the equality  the equality $(L/L^{pi})^{sep}=L^{pi}L^{sep}/L^{pi}$.

 (x) $m^*=m$: By the definition of the number   $m^*$, 
$$ 
\th^\frac{1}{p^{n-m'}}\not\in L^{pi}\t L^{sep}\;\; \text{for all $m'$ such that $ m^*<m'\leq n$. }
$$
By the statements (v) and (viii),   $f^{sep,\frac{1}{p^{m^*}}}(x)\in \Irr_m(M^*[x])=\Irr_m(L^{pi}[x])$. By the definition of the number $m=m_f$, 
 $$
 f^{sep,\frac{1}{p^{m}}}(x):=\sum_{i=0}^s\l_i^\frac{1}{p^m} x^i\in L^{pi}[x].  
 $$
 Therefore,  $m^*\leq m$ (by the maximality of $m$). Since $f(x)\in \Irr_m(K[x])$,  $f(x)=\Big(f^{sep,\frac{1}{p^{m}}}(x^{p^{n-m}}) \Big)^{p^m}$ and $\Big( L^{pi}\Big)^{p^{m^*}}\subseteq K$ (the statement (viii)), the polynomial 
 $f^{sep,\frac{1}{p^{m}}}(x)\in L^{pi}[x]$
 is an {\em irreducible} polynomial over the field $L^{pi}$. Otherwise,  $f^{sep,\frac{1}{p^{m}}}(x)=a(x)b(x)$ for some non-scalar polynomials  $a(x), b(x)\in L^{pi}[x]$, and so 
 $$
 f(x)=\bigg(f^{sep,\frac{1}{p^{m}}}(x^{p^{n-m}}) \bigg)^{p^m}=\bigg(a(x^{p^{n-m}})b(x^{p^{n-m}}) \bigg)^{p^m}=a(x^{p^{n-m}})^{p^m}b(x^{p^{n-m}})^{p^m}
 $$
where  $a(x^{p^{n-m}})^{p^m}, b(x^{p^{n-m}})^{p^m}\in K[x]\backslash K$, a contradiction. 

The polynomial 
 $f^{sep,\frac{1}{p^{m}}}(x)\in L^{pi}[x]$
 is a {\em separable} polynomial over the field $L^{pi}$ (since it is an irreducible polynomial over $L^{pi}$ and its derivative is a nonzero polynomial). The equality
 $$0=f(\th)=\Big(f^{sep,\frac{1}{p^{m}}}(\th^{p^{n-m}}) \Big)^{p^m}$$
implies the equality  $f^{sep,\frac{1}{p^{m}}}(\th^{p^{n-m}})=0$. This means that 
$$
\th^{p^{n-m}}\in (L/L^{pi})^{sep}\stackrel{{\rm (ix)}}{=}L^{pi}L^{sep}/L^{pi}.
$$
 Now, we must have $m^*=m$ (since otherwise, $m^*<m$ and  $\th^{n-m}\not\in L^{pi}L^{sep}$, a contradiction).

(xi) $\max\{m'\in \N \, | \, \th^{p^{n-m'}}\in L^{pi}L^{sep} \}=\max\{m'\in \N \, | \, L^{p^{n-m'}}\subseteq  L^{pi}L^{sep} \}$:
 By statement 1, the field extension $L/L^{pi}L^{sep}=K(\th)/L^{pi}L^{sep}$ is a purely inseparable field extension since $\th^{p^n}\in L^{sep}$. Now, the equality  in the statement (xi) follows from the equality $L=K(\th)$.

 (xii) {\em The number $m$ is an isomorphism invariant of the field extension} $L/K$: By the definition, the number $m^*=\max\{m'\in \N \, | \, L^{p^{n-m'}}\subseteq  L^{pi}L^{sep} \}$ (the statement (xi)) is an isomorphism invariant of the field extension $L/K$. Hence, so is the number $m=m^*$ (the statement (x)). 
\end{proof}

{\bf The invariant $m_f$ and the maximal purely inseparable subfield $L^{pi}$ of $L$.}  By Theorem \ref{27Apr26}.(2), $L^{pi}=K\Big(\l_0^\frac{1}{p^{m_f}}, \ldots ,  \l_{s-1}^\frac{1}{p^{m_f}}\Big)$. Therefore, the invariant  $m_f$ uniquely determines the field $L^{pi}$. 
For  natural numbers $m\geq 1$ and $s\geq 0$, let 
$$
\N_{<p^m}^s:=\Big\{ \alpha =(\alpha_0,\alpha_1, \ldots , \alpha_{s-1})\, \Big| \, 0\leq \alpha_i< p^m\; {\rm for}\;\; i=0,1,\ldots, s-1  \Big\}.
$$
Lemma \ref{a17May26} provides two different methods for determining the invariant $m_f$.

\begin{lemma}\label{a17May26}
Suppose that $K$ is a field of prime characteristic $p>0$, $L/K$ is a simple finite  field extension and $L=K(\th)=K[x]/(f(x))$ where $f(x)=f^{sep}(x^{p^{n}})=\sum_{i=0}^s\l_ix^{ip^n}\in \Irr_m(K[x])$, $\l_i\in K$ and  $\l_s=1$. Then: 
\begin{enumerate}

\item The number $m_f$ is the maximal number $m\in \{ 0,1,\ldots , n\}$ such that 
there exists (necessarily unique)  elements $\l_{m, i,j}\in K$ such that 
$$ \l_i^\frac{1}{p^m}=\sum_{j=0}^{sp^n-1}\l_{m,i,j}\th^j,\;\;i=0,1,\ldots , s-1,
$$
or, equivalently, 
$$ \l_i=\sum_{j=0}^{sp^n-1}\l_{m,i,j}^{p^m}\th^{jp^m},\;\;i=0,1,\ldots , s-1.
$$

\item The number $m_f$ is the maximal number $m\in \{ 0,1,\ldots , n\}$ such that 
there exist  elements $\g_{m, i}=\sum_{\alpha \in \N_{<p^m}^s} \g_{m,i,\alpha} \l^\frac{\alpha}{p^m}$, $i=0,1,\ldots , s-1$, where $\g_{m,i,\alpha}\in K$, $\alpha =(\alpha_0,\alpha_1, \ldots , \alpha_{s-1})$ and $\l^\frac{\alpha}{p^m}:=\prod_{j=0}^{s-1}\l_j^\frac{\alpha_j}{p^m}$,  such that 
$$
\th^{p^{n-m}}=\sum_{i=0}^{s-1}\g_{m,i}\th^{ip^n},
$$
or, equivalently,
$$
\th^{p^n}=\sum_{i=0}^{s-1}\g_{m,i}^{p^m}\th^{ip^{n+m}}
$$
where $\g_{m,i}^{p^m}=\sum_{\alpha \in \N_{<p^m}^s} \g_{m,i,\alpha}^{p^m} \l^{\alpha}\in K$ and $\l^{\alpha} =\prod_{j=0}^{s-1}\l_j^{\alpha_j}$.

\end{enumerate}
\end{lemma} 

\begin{proof} 1.  By the definition, $m_f=\max\bigg\{ m=0,1,\ldots , n\, \bigg| \, \l_i^\frac{1}{p^{m}}\in L$ for all $i=0, \ldots , s-1\bigg\}$. Now, statement 1 follows from  the equalities   $L=\bigoplus_{i=0}^{sp^n-1}K\th^i$ and  $L^{pi}=K\Big(\l_0^\frac{1}{p^{m_f}}, \ldots ,  \l_{s-1}^\frac{1}{p^{m_f}}\Big)$ (Theorem \ref{27Apr26}.(2)).

2.  By Theorem \ref{27Apr26}.(1,3) and Theorem \ref{27Apr26}.(2), 
\begin{eqnarray*}
L^{pi}L^{sep}&=&L^{pi}\t L^{sep}=L^{pi}\t K(\th^{p^n})=\bigoplus_{i=0}^{s-1}L^{pi}\th^{ip^n},\\
L^{pi}&=&K\Big(\l_0^\frac{1}{p^{m_f}}, \ldots ,  \l_{s-1}^\frac{1}{p^{m_f}}\Big)=\sum_{\alpha\in \N_{<p^{m_f}}}K\l^\frac{\alpha}{p^{m_f}},
\end{eqnarray*}
respectively. By Theorem \ref{27Apr26}.(4), the number $m_f$ is the maximal number $m\in \{ 0,1,\ldots , n\}$ such that $\th^{p^{n-m}}\in L^{pi}L^{sep}$. Now, statement 2 follows. 
\end{proof}

In statement 1, for each $i=0,1,\ldots , s-1$,  the coefficients $\l_{m,i,j}^{p^m}\in K^{p^n}\subseteq K$ in the equation 
$$
\l_i=\sum_{j=0}^{sp^n-1}\l_{m,i,j}^{p^m}\th^{jp^m}
$$
 are exactly the solutions to the linear system generated by reducing the sum
modulo $f(\theta) = 0$ using the relations established by $\lambda_i$.

Similarly, in statement 2,  the coefficients $\g_{m,i}^{p^m}\in K^{p^n}\subseteq K$ in the equation 
$$
\th^{p^n}=\sum_{i=0}^{s-1}\g_{m,i}^{p^m}\th^{ip^{n+m}}
$$
 are exactly the solutions to the linear system generated by reducing the sum
modulo $f(\theta) = 0$ using the relations established by $\lambda_i$.\\

{\bf Criterion for $L^{pi}=K$ for a simple field extension $L/K$.} For a simple field extension $L/K$, Corollary  \ref{a7May26} is  an explicit criterion for $L^{pi}=K$.
 

\begin{corollary}\label{a7May26}
Suppose that $K$ is a field of prime characteristic $p>0$, $L/K$ is a simple finite  field extension and $L=K(\th)=K[x]/(f(x))$ where $f(x)=f^{sep}(x^{p^{n}})=\sum_{i=0}^s\l_ix^{ip^n}\in \Irr_m(K[x])$, $\l_i\in K$ and  $\l_s=1$. Then the following statements are equivalent: 
\begin{enumerate}

\item $L^{pi}=K$.

\item Either $n=0$ or $n\geq 1$ and $\l_i^\frac{1}{p}\not\in L$ for some index $i\in \{ 0,1,\ldots,  s-1\}$. 

\item $m_f=0$.

\item $[L:L^{sep}]=n$.

\end{enumerate}
\end{corollary} 

\begin{proof}  By Theorem \ref{27Apr26} or diagram (\ref{L=LLL}),  $L^{pi}=K$ iff  $m_f=0$ iff either $n=0$ or $n\geq 1$ and $\l_i^\frac{1}{p}\not\in L$ for some index  $i\in \{ 0,1,\ldots,  s-1\}$ and $m_f=0$ iff $[L:L^{sep}]=n$.
\end{proof}

\begin{example}[An example where  $L^{pi}=K$]  Let $p > 2$ be a prime,  $K = \mathbb{F}_p(u, v)$ be the  field  of rational functions in two variables over $\mathbb{F}_p$ and  $L = K(\th)$ where $\th$ is a root of the irreducible polynomial
$$ 
f(x) = x^{2p} + ux^p + v.
$$
Clearly, $f^{sep}(x)=x^2 + ux + v$ and $n:=\deg_{\rm ins}(f)=1$. Since $u^\frac{1}{p}\not \in L$, we have that $m_f=0$ and  then,  by Corollary  \ref{a7May26}, $L^{pi}=K$. 
\end{example}

{\bf Criteria for $L = L^{pi}L^{sep}$ where $L/K$ is a simple finite field extension.}  
For a simple field extension $L/K$ of prime  characteristic $p>0$, Theorem \ref{B27Apr26} presents new explicit criteria for $L=L^{pi}L^{sep}$. 

\begin{theorem}\label{B27Apr26}
Suppose that $K$ is a field of prime characteristic $p>0$,  $L/K$ is a simple finite  field extension and $L=K(\th)=K[x]/(f(x))$ where $f(x)=f^{sep}(x^{p^{n}})=\sum_{i=0}^s\l_ix^{ip^n}\in \Irr_m(K[x])$, $\l_i\in K$ and  $\l_s=1$. Then the following statements are equivalent:

\begin{enumerate}

\item $L = L^{pi}L^{sep}\Big(=L^{pi}\t L^{sep}\Big)$.
 
\item  $\l_i^\frac{1}{p^n}\in L$ for all $i=1, \ldots , s-1$, i.e. $m_f=n$.

\item  $L^{pi}=K\Big( \l_0^\frac{1}{p^n}, \ldots , \l_{s-1}^\frac{1}{p^n} \Big)$. 

\end{enumerate}
\end{theorem}

\begin{proof} $(1\Leftrightarrow 2)$ By Theorem \ref{27Apr26}.(4), $[L:L^{pi}\t L^{sep}]=p^{n-m_f}$ and the result follows.

$(2\Leftrightarrow 3)$ By Theorem \ref{27Apr26}.(2), the equality $m_f=n$ is equivalent to the equality $L^{pi}=K\Big( \l_0^\frac{1}{p^n}, \ldots , \l_{s-1}^\frac{1}{p^n} \Big)$. 
\end{proof}

\begin{example}[A counterexample where $L \neq L^{pi}L^{sep}$]
Let $p > 2$ be a prime,  $K = \mathbb{F}_p(u, v)$ be the  field  of rational functions in two variables over $\mathbb{F}_p$ and  $L = K(\th)$ where $\th$ is a root of the irreducible polynomial
$$ 
f(x) = x^{2p} + ux^p + v.
$$
Clearly, $f^{sep}(x)=x^2 + ux + v$ and $n:=\deg_{\rm ins}(f)=1$. Since $u^\frac{1}{p}\not \in L$, we have that $m_f=0\neq 1=n$ and  then,  by Theorem \ref{B27Apr26}, $L \neq L^{pi}L^{sep}$. 
\end{example}


\section{Explicit descriptions of the subfields $(NL)^{pi}$ and $(NL)^{pi}(NL)^{sep}$  of $NL/N$ where $L/K$ is a simple field extension and $N/K$ is a purely inseparable field extension}\label{NL=NLPI}

In this section, for a simple field extension $L/K=K(\th)/K$ and a purely inseparable field extension $N/K$ (not necessarily finite),
 Theorem \ref{A6May26} describes the structure of the compositum $NL$, the degree $[NL:N]$ and the minimal polynomial for the simple field extension $NL/N$. 
 Theorem \ref{VB-27Apr26} yields explicit descriptions of the following subfields of  $NL/N$ --  $(NL)^{pi}/N$, $(NL)^{sep}/N$ and their  compositum $(NL)^{pi}(NL)^{sep}/N$ --  in terms of the coefficients of the minimal  polynomial $f$ of the element  $\th\in L$ over $K$ and two natural numbers (that are field invariants) $m_{f,N}$ and $m_{f,N(\th)}$ associated with  $f$,  $N$ and $NL$. We compute explicit numerical values for the following degrees:  $[(NL)^{pi}:N]$, $[(NL)^{sep}:N]$, $[(NL)^{pi}(NL)^{sep}:(NL)^{pi}]$,  $[(NL)^{pi}(NL)^{sep}:(NL)^{sep}]$  and $[L:(NL)^{pi}(NL)^{sep}]$.   Corollary  \ref{Na7May26} is  an explicit criterion for $(NL)^{pi}=N$. Theorem \ref{NB27Apr26} is an explicit criteria for $NL=(NL)^{pi}(NL)^{sep}$. \\

{\bf The structure of the field extension $NL=N (\th)$  where $N/K$ is a purely inseparable field extension.}  Suppose that $L=K(\th)=K[x]/(f)$ is a simple field extension, where $\th \in L$,  and the polynomial  $f(x)=f^{sep}(x^{p^n})=\sum_{i=0}^s\l_ix^{ip^n}\in \Inn_m(K[x])$ is the minimal polynomial of $\th$. 

\begin{definition}\label{Def=mfN}
For a purely inseparable field extension $N/K$, let 
\begin{eqnarray*}
m_{f,N}&:=&\max\bigg\{ m'=0,1,\ldots , n\, \bigg| \, \l_i^\frac{1}{p^{m'}}\in N\;\; \text{ for all}\;\;i=0, \ldots , s-1\bigg\},\\
f_N(x)&:=&\sum_{i=0}^s\l_i^\frac{1}{p^{m_{f,N}}}x^{ip^{n-m_{f,N}}}\in N[x],\\
f^{sep}_N(x)&:=&\sum_{i=0}^s\l_i^\frac{1}{p^{m_{f,N}}}x^i\in N[x],\\
M_{f,N}&:=&K\bigg(\l_0^\frac{1}{p^{m_{f,N}}}, \ldots ,  \l_{s-1}^\frac{1}{p^{m_{f,N}}}\bigg).
\end{eqnarray*}
\end{definition}
Notice that the field extension $M_{f,N}/K$ is a purely inseparable  finite field extension, 
 $M_{f,N}\subseteq N$ and  $M_{f,N}=M_{f,M_{f,N}}$.

\begin{proposition}\label{6May26}
Suppose that $K$ is a field of prime characteristic $p>0$, $L/K$ is a simple finite  field extension and $L=K(\th)=K[x]/(f(x))$ where $f(x)=f^{sep}(x^{p^{n}})=\sum_{i=0}^s\l_ix^{ip^n}\in \Irr_m(K[x])$, $\l_i\in K$, $\l_s=1$ and  $N/K$ is a purely inseparable field extension such that $N^{p^{m_{f,N}}}\subseteq K$. Then: 
\begin{enumerate}

\item $f_N\in \Irr_m(N[x])$, $f_N\in \Irr_m (M_{f,N}[x])$, $f= f_N^{p^{m_{f,N}}}$,   $\deg (f_N)=sp^{n-m_{f,N}}$ where $M_{f,N}=K\bigg(\l_0^\frac{1}{p^{m_{f,N}}}, \ldots ,  \l_{s-1}^\frac{1}{p^{m_{f,N}}}\bigg)$.

\item $N(\th)\simeq N[x]/(f_N)$ and  $[N(\th):N]=\deg (f_N)=sp^{n-m_{f,N}}$.

\item $M_{f,N}(\th)\simeq M_{f,N}[x]/(f_N)$ and  $[M_{f,N}(\th):M_{f,N}]=\deg (f_N)=sp^{n-m_{f,N}}$.

\end{enumerate}
\end{proposition}

\begin{proof}  1. The equality  $f_= f_N^{p^{m_{f,N}}}$ is obvious.

Suppose that $f_N\not\in \Irr_m(N[x])$ and 
we seek a contradiction. Then $f_N(x)=ab$ for some non-scalar polynomials $a,b\in N[x]$. Then $a^{p^{m_{f,N}}}, b^{p^{m_{f,N}}}\in K[x]$ (since $N^{p^{m_{f,N}}}\subseteq K$) and 
$$ 
f= f_N^{p^{m_{f,N}}}=\Big(ab \Big)^{p^{m_{f,N}}}=a^{p^{m_{f,N}}}b^{p^{m_{f,N}}}.
$$
Therefore, $f\not\in \Irr_m(K[x])$, a contradiction. Thus, $f_N\in \Irr_m(N[x])$, and so $f_N\in \Irr_m\bigg(K\bigg(\l_0^\frac{1}{p^{m_{f,N}}}, \ldots ,  \l_{s-1}^\frac{1}{p^{m_{f,N}}}\bigg)\bigg)$ (since $K\bigg(\l_0^\frac{1}{p^{m_{f,N}}}, \ldots ,  \l_{s-1}^\frac{1}{p^{m_{f,N}}}\bigg)\subseteq N$).

2. Statement 2 follows from statement 1.

3. Statement 3 follows from statement 1.
\end{proof}

For the irreducible  polynomial $f(x)=f^{sep}(x^{p^{n}})=\sum_{i=0}^s\l_ix^{ip^n}\in \Irr_m(K[x])$ where  $\l_i\in K$ and  $\l_s=1$, there is a tower of purely inseparable finite field extensions
\begin{equation}\label{Nfi-tower}
N_{f,0}:=K\subset N_{f,1}\subset \cdots \subset N_{f,i}\subset \cdots \subset N_{f,n}
\;\; {\rm where}\;\; N_{f,i}:=K\bigg(\l_0^\frac{1}{p^i}, \ldots ,  \l_{s-1}^\frac{1}{p^i}\bigg).
\end{equation}
Indeed, since $f(x)=f^{sep}(x^{p^{n}})=\sum_{i=0}^s\l_ix^{ip^n}\in \Irr_m(K[x])$, we must have the  proper inclusion $K\subset N_{f,1}$ (provided $n\geq 1$) which implies the proper inclusions in the tower of  subfields (by the definition of the fields $N_{f,i}$). 
 Therefore,  
 
\begin{equation}\label{Nfi-exp}
\exp (N_{f,i}/K)=i\;\; \text{ for all}\;\;i=0, 1, \ldots , n
\end{equation}
 where $\exp (N_{f,i}/K)$ is the {\em exponent} of the purely inseparable field extension $N_{f,i}/K$ (i.e.  $i$ is the minimal natural number  such that $N_{f,i}^{p^i}\subseteq K$). In particular, 
\begin{equation}\label{Nfi}
N_{f,i}^{p^i}\subseteq K\;\; \text{ for all}\;\; i=0, 1, \ldots , n.
\end{equation}
Notice, that the polynomial $f$ is a separable polynomial iff $n=0$ iff the tower of subfields in (\ref{Nfi-tower}) consists of the single subfield $K$. 
 Corollary \ref{a6May26} describes the field extensions $N_{f,i}(\th)$ where $i=0,1,\ldots, n$.

\begin{corollary}\label{a6May26}
Suppose that $K$ is a field of prime characteristic $p>0$, $L/K$ is a simple finite  field extension and $L=K(\th)=K[x]/(f(x))$ where $f(x)=f^{sep}(x^{p^{n}})=\sum_{i=0}^s\l_ix^{ip^n}\in \Irr_m(K[x])$, $\l_i\in K$, $\l_s=1$ and  $N_i=N_{f,i}=K\bigg(\l_0^\frac{1}{p^i}, \ldots ,  \l_{s-1}^\frac{1}{p^i}\bigg)$ for $i=0,1,\ldots , n$. Then: 
\begin{enumerate}

\item $f_{N_i}=\sum_{j=0}^s\l_j^\frac{1}{p^i}  x^{ip^{n-i}}\in \Irr_m(N_i[x])$, $f= f_{N_i}^{p^i}$ and  $\deg (f_{N_i})=sp^{n-i}$.

\item $N_i(\th)\simeq N_i[x]/(f_{N_i})$ and  $[N_i(\th):N_i]=\deg (f_N)=sp^{n-m_{f,N}}$.

\end{enumerate}
\end{corollary}

\begin{proof} By (\ref{Nfi-exp}), $m_{f,N_i}=i$ for all $i=0,1,\ldots, n$. Now,  by (\ref{Nfi}), the corollary follows from Proposition \ref{6May26}. 
\end{proof}

The following lemma is used  in the proof of Theorem \ref{A6May26}.

\begin{lemma}\label{a9May26}

Suppose that $K$ is a field of prime characteristic $p>0$ and $\phi\in \Irr (K[x])$ is an irreducible  separable polynomial. Then  $\phi\in \Irr (N[x])$  is an irreducible  separable polynomial for all purely inseparable field extensions $N/K$.
\end{lemma}

\begin{proof}  Suppose that the polynomial $\phi$ is a reducible polynomial over the field $N$. Then $\phi = ab$ for some nonscalar polynomials $a,b\in N[x]$. By the assumption the polynomial $\phi \in K[x]$ is a separable polynomial over $K$. Therefore, $\phi = \prod_{i=1}^d(x-\th_i)$ where $d=\deg (\phi)$ and $\th_1, \ldots , \th_d\in \bK^{sep}$ are distinct roots of the polynomial  $\phi$ (by the separability of $\phi$). Then, up to order of the roots of $\phi$, 
$$
a=\prod_{i=1}^n (x-\th_i)=\sum_{i=0}^n (-1)^is_i(\th_1, \ldots ,\th_n)x^{n-i} \
$$
where $s_i (x_1, \ldots , x_n)$ is the elementary symmetric polynomial/function of degree $i$ in the variables $x_1, \ldots , x_n$ for $i=1,\ldots, n$ and $s_0 (x_1, \ldots , x_n):=1$. Since $\th_1, \ldots , \th_d\in \bK^{sep}$ and $a\in N[x]$, the coefficients of the polynomial $a$ belong to the intersection 
$\bK^{sep}\cap N=K$, i.e. $a\in K[x]$. By symmetry, $b\in K[x]$. Therefore, the polynomial $f=ab$ is a reducible polynomial over $K$, a contradiction. 
\end{proof}

For a  purely inseparable field extension $N/K$, Theorem \ref{A6May26} describes the structure of the compositum $NL=N(\th)$.

\begin{theorem}\label{A6May26}
Suppose that $K$ is a field of prime characteristic $p>0$, $L/K$ is a simple finite  field extension and $L=K(\th)=K[x]/(f(x))$ where $f(x)=f^{sep}(x^{p^{n}})=\sum_{i=0}^s\l_ix^{ip^n}\in \Irr_m(K[x])$, $\l_i\in K$, $\l_s=1$ and  $N/K$ is a purely inseparable field extension. Then: 
\begin{enumerate}

\item $f_N=\sum_{i=0}^s\l_i^\frac{1}{p^{m_{f,N}}}x^{ip^{n-m_{f,N}}}\in \Irr_m(N[x])$, $f_N\in \Irr_m (M_{f,N}[x])$, $f= f_N^{p^{m_{f,N}}}$,   $\deg (f_N)=sp^{n-m_{f,N}}$ where $M_{f,N}=K\bigg(\l_0^\frac{1}{p^{m_{f,N}}}, \ldots ,  \l_{s-1}^\frac{1}{p^{m_{f,N}}}\bigg)$.

\item $N(\th)\simeq N[x]/(f_N)$ and  $[N(\th):N]=\deg (f_N)=sp^{n-m_{f,N}}$.


\end{enumerate}
\end{theorem}

\begin{proof} 1. Clearly, $g:=\sum_{i=0}^s\l_i^\frac{1}{p^n}x^i\in N_{f,n}[x]$. 

(i) {\em The polynomial $g\in \Irr_m(N_{f,n}[x])$ is a monic  irreducible  separable polynomial over the field $N_{f,n}$}:  By Corollary \ref{a6May26}.(1), 
$$
g=f_{N_{f,n}}\in \Irr_m(N_{f,n}[x]).$$
Therefore, the polynomial $g$ is a separable polynomial over the field $N_{f,n}$ (since the polynomial $g$ is an irreducible polynomial over $N_{f,n}$ and 
$\frac{dg}{dx}\neq 0$). 

(ii) {\em For all purely inseparable field extensions $N/K$ such that $N_{f,n}\subseteq N$, the polynomial $g\in \Irr_m(N[x])$ is a monic  irreducible  separable polynomial over the field $N$}: The statement (ii) follows from the statement (i) and Lemma \ref{a9May26}.

(iii) {\em For each} $j=0,1,\ldots , n$, $K({\rm coef} (g^{p^j}))=N_{f,n-j}$: The statement (iii) follows from the equality
$$
g^{p^j}=\sum_{i=0}^s\l_i^\frac{1}{p^{n-j}}x^{ip^{n-j}}.
$$
(iv) $K\Big({\rm coef}\Big(g^{i_\d p^\d}\Big)\Big) N_{f,n-\d-1}=N_{f,n-\d}$ {\em for all natural numbers} $i=i_\d p^\d$ {\em where} $i_\d=1,\ldots , p-1$ {\em and} $ \d =0,1,\ldots , n$: By the statement (iii) and the inclusion $N_{f,n-\d}\supseteq N_{f,n-\d-1}$, 
$$
N_{f,n-\d}=K\Big({\rm coef} \Big(g^{p^\d}\Big)\Big)  \supseteq K\Big({\rm coef}\Big(g^{i_\d p^\d}\Big)\Big) N_{f,n-\d-1}.
$$
 Since $\gcd (i_\d, p)=1$, $\alpha i_\d-\beta p=1$ for some integers $\alpha,\beta  \geq 1$ (take $\alpha\in \{ 0,1, \ldots, p-1\}$ such that $\alpha =i_\d^{-1}\in \F_p$). It follows from the equalities and the inclusions, $$g^{p^\d}=g^{1\cdot p^\d}=g^{(\alpha i_\d-\beta p)p^\d}=\frac{g^{\alpha i_\d p^\d}}{g^{\beta p^{\d +1}}}, \;\;  {\rm coef}\Big(g^{\alpha i_\d p^\d}\Big)\subseteq {\rm coef}\Big(g^{ i_\d p^\d}\Big)
    \;\; {\rm and}\;\;   
 {\rm coef}\Big(g^{\beta p^{\d +1}}\Big)\subseteq N_{n-\d -1},$$
that 
$
N_{f,n-\d}=K({\rm coef} \Big(g^{p^\d}\Big)\Big)  \subseteq K\Big({\rm coef}\Big(g^{i_\d p^\d}\Big)\Big) N_{f,n-\d-1},
$
and the statement (iv) follows.

(v) {\em For all natural numbers} $i=\sum_{\nu=0}^\mu i_\nu p^\nu$ {\em where} $i_\nu \in \{ 0,1,\ldots , p-1\}$ {\em and} $\mu \leq n$,
$$
K\Big( {\rm coef}(g^i)\Big)N_{f,n-\d-1}=K\Big({\rm coef}\Big(g^{i_\d p^\d}\Big)\Big)N_{f,n-\d-1}=N_{f,n-\d}\;\; {\it where}  \;\; \d :=\min \{\nu \, | \, i_\nu\neq 0\}:
$$ 
It suffices to show that the first equality holds since the second one is the statement (iv). In view of the statement (iv), we may assume that $\d <\mu$. By the statement (iii), 
$$
g^i=\prod_{\nu=\d}^\mu \Big(g^{p^\nu}\Big)^{i_\nu}\in g^{i_\d p^\d}\prod_{\nu>\d}^\mu  N_{f,n-\nu}[x]  \subseteq 
g^{i_\d p^\d} N_{f,n-\d-1}[x]
$$
since $N_{f,n-\d}\supset N_{f,n-\d -1}\supset \cdots \supset N_{f,-1}\supset N_{f,0}=K$. 
Therefore, $K\Big({\rm coef}(g^i)\Big)\subseteq K\Big({\rm coef}\Big(g^{i_\d p^\d}\Big)\Big) N_{f,n-\d-1}$. Hence, 
$$
K\Big({\rm coef}(g^i)\Big)N_{f,n-\d-1}\subseteq K\Big({\rm coef}\Big(g^{i_\d p^\d}\Big)\Big) N_{f,n-\d-1}.
$$
The equality $g^{i_\d p^\d}=\frac{g^i}{\prod_{\nu>\d}^\mu \Big(g^{p^\nu}\Big)^{i_\nu}}$ implies the inclusion 
 $K\Big({\rm coef}\Big(g^{i_\d p^\d}\Big)\Big)\subseteq K\Big({\rm coef}(g^i)\Big)N_{f,n-\d-1}$. Hence, 
 $$
 K\Big({\rm coef}\Big(g^{i_\d p^\d}\Big)\Big) N_{f,n-\d-1}\subseteq K\Big({\rm coef}(g^i)\Big)N_{f,n-\d-1},
 $$
 and the statement (v) follows. 

Let $N/K$ be a purely inseparable field extension. By the definition of the natural number $m:=m_{f,N}$ and the field $M_{f,N}$, 
$$
N_{f,m}=M_{f,N}\subseteq N.
$$
  (vi) $f_N=\sum_{i=0}^s\l_i^\frac{1}{p^m}x^{ip^{n-m}}=f_{N_{f,m}}\in \Irr_m(N_{f,m}[x])$: The statement (vi) follows from   Corollary \ref{a6May26}.(1).\\

(vii) $f_N\in \Irr_m(N[x])$: The polynomial $f_N$ is a monic polynomial. Suppose that the polynomial $f_N$ is a reducible polynomial over the field $N$, i.e. $f=ab$ for some non-scalar polynomials $a,b\in N[x]$. We seek a contradiction. Notice that
 $$
f_N=g^{p^{n-m}}\;\; {\rm  and }\;\; g\in \Irr_m\Big(\bK^{pi}[x]\Big)\;\; \text{(by the statement (ii)). }
$$
Therefore, $a=g^i$ and $b=g^j$ for some natural numbers $i\geq 1$ and $j\geq 1$ such that $i+j=p^{n-m}$. Therefore, $i=\sum_{\nu=0}^\mu i_\nu p^\nu$ where $i_\nu \in \{ 0,1,\ldots , p-1\}$ and $\mu <n-m$.
Let $\d :=\min \{\nu \, | \, i_\nu\neq 0\}$. 
 Then $\d< n-m$ or, equivalently, $m<n-\d $. By the statement (v), 
$$
K\Big( {\rm coef}(g^i)\Big)N_{f, n-\d-1}=N_{f, n-\d}\supset N_{f, n-\d}\supseteq N_{f, m}.
$$
Since $g^i\in N[x]$, we must have $K\Big( {\rm coef}(g^i)\Big)\subseteq N_{f, m}$. Therefore,
$$
N_{f, n-\d}=K\Big( {\rm coef}(g^i)\Big)N_{f, n-\d-1}\subseteq N_{f, m}N_{f, n-\d-1}=N_{f, n-\d-1},
$$
 a contradiction.

2. Statement 2 follows from statement 1.
\end{proof}

\begin{example}\label{f=2nml} Let $p > 2$ be a prime,  $K = \mathbb{F}_p(u, v)$ be the  field  of rational functions in two variables over $\mathbb{F}_p$ and  $L = K(\th)$ where $\th$ is a root of the irreducible polynomial
$$ 
f(x) = x^{2p^n} + ux^{p^n}+ v.
$$
Clearly, $f^{sep}(x)=x^2 + ux + v$ and $n=\deg_{\rm ins}(f)$. Let $N=\mathbb{F}_p(u^\frac{1}{p^m}, v^\frac{1}{p^l})$ for some natural numbers $m$ and $l$ such that $1\leq m \leq l$ and $m\leq n$. Then, by Theorem \ref{A6May26},  
$$
f_N= x^{2p^{n-m}} + u^\frac{1}{p^m}x^{p^{n-m}}+ v^\frac{1}{p^m}, \;\; f_N^{sep}=x^2 + u^\frac{1}{p^m}x+ v^\frac{1}{p^m}, \;\; M_{f,N}=\mF_p (u^\frac{1}{p^m}, v^\frac{1}{p^m}),
$$ 
$m_{f,N}=m$, $N(\th)=N[x]/(f_N)$ and $[N(\th):N]=2p^{n-m}$. 
\end{example}

{\bf Explicit descriptions of the subfields $(NL)^{pi}$ and $(NL)^{pi}(NL)^{sep}$  of  $NL/N$ where $N/K$ is a purely inseparable field extension.}

\begin{definition}\label{DEF=mfNth}
For a purely inseparable field extension $N/K$, let $m_{f,N(\th)}:=m_{f,N(\th)/N}:=m_{f,NL/N}$,
\begin{eqnarray*}
m_{f,N(\th)}&:=&\max\bigg\{ m'=0,1,\ldots , n-m_{f,N}\, \bigg| \, \l_i^\frac{1}{p^{m_{f,N}+m'}}\in NL\;\; \text{ for all}\;\;i=0, \ldots , s-1\bigg\}\\
&=&\max\bigg\{ m'=0,1,\ldots , n-m_{f,N}\, \bigg| \, \l_i^\frac{1}{p^{m_{f,N}+m'}}\in (NL/N)^{pi}\;\; \text{ for all}\;\;i=0, \ldots , s-1\bigg\},\\
f_{NL}(x)&:=&\sum_{i=0}^s\l_i^\frac{1}{p^{m_{f,N}+m_{f,N(\th)}}} x^{ ip^{ n-m_{f,N}-m_{f,N(\th)} } }\in (NL/N)^{pi}[x]  \subseteq  NL[x],\\
f^{sep}_{NL}(x)&:=&\sum_{i=0}^s\l_i^\frac{1}{p^{m_{f,N}+m_{f,N(\th)}}} x^i\in (NL/N)^{pi}[x]  \subseteq  NL[x],\\
M_{f,LN}&:=&N\bigg(\l_0^\frac{1}{p^{m_{f,N}+m_{f,N(\th)}}}, \ldots ,  \l_{s-1}^\frac{1}{p^{m_{f,N}+m_{f,N(\th)}}} \bigg).
\end{eqnarray*}
\end{definition}
Clearly, $f_N=f_{NL}^{p^{m_{f,N(\th)}}}$, $M_{f,N}\subseteq M_{f,LN}$ and $[M_{f,N}:M_{f,LN}]=p^{m_{f,N(\th)}}$.
 Since  $N/K$ is a purely inseparable field extension,
\begin{equation}\label{NLKpi=NLNpi}
\Big(NL/N \Big)^{pi}=\Big(NL/K \Big)^{pi}.
\end{equation}
By Theorem \ref{VB-27Apr26}, there is a field diagram where the edges are labelled by the degrees of the corresponding field extensions over the field $N$, see Theorem \ref{VB-27Apr26} for details.
\begin{equation}\label{NL=NLLL}
\begin{tikzpicture}[scale=2.8, every node/.style={font=\normalsize}]

\node (L) at (0,2.6) {$NL=N(\th)$};
\node (A) at (0,0) {$N$};
\node (B) at (-1.2,0.9) {$(NL)^{pi}=N\bigg( 
\l_0^\frac{1}{p^{m_{f,N}+m_{f,N(\th)}}}, \ldots ,  \l_{s-1}^\frac{1}{p^{m_{f,N}+m_{f,N(\th)}}}\bigg)$};
\node (D) at (1.2,0.9) {$(NL)^{sep}=NL^{sep}=N\Big(\th^{p^{n-m_{f,N}}}\Big)$};
\node (C) at (0,1.7) {$(NL)^{pi}(NL)^{sep} \;=\; (NL)^{pi}\!\otimes_N (NL)^{sep}=(NL)^{pi}\Big(\th^{p^{n-m_{f,N}-m_{f,N(\th)}}}\Big)$};

\draw (L) -- node[right] {$p^{\,n-m_{f,N}-m_{f,N(\th)} }$} (C);

\draw (C) -- node[right] {$p^{{m_{f,N(\th)}}}$} (D);
\draw (D) -- node[right] {$\deg(f^{sep}_N)$} (A);
\draw (A) -- node[left]  {$p^{{m_{f,N(\th)}}}$} (B);
\draw (B) -- node[left]  {$\deg(f^{sep}_N)$} (C);

\end{tikzpicture}
\end{equation}

For the simple field extension $L/K$, Theorem \ref{VB-27Apr26} gives explicit descriptions of the  subfields $(NL)^{pi}/N$,   $(NL)^{sep}/N$ and $(NL)^{pi}(NL)^{sep}/N$ in terms of the coefficients of the polynomial $f_N$ and  the numbers $m_{f,N}$ and $m_{f,N(\th )}$. It also clarifies why, in general,  the field $(NL)^{pi}$ properly contains the field  $NL^{pi}$. Theorem \ref{VB-27Apr26} yields a  criteria for $(NL)^{pi}=NL^{pi}$ 
 (Corollary \ref{Na7May26}).


\begin{theorem}\label{VB-27Apr26}
Suppose that $K$ is a field of prime characteristic $p>0$, $L/K$ is a simple finite  field extension and $L=K(\th)=K[x]/(f(x))$ where $f(x)=f^{sep}(x^{p^{n}})=\sum_{i=0}^s\l_ix^{ip^n}\in \Irr_m(K[x])$, $\l_i\in K$, $\l_s=1$   and $N/K$ is a purely inseparable field extension. Then (below $(NL)^{sep}:=(NL/N)^{sep}$ and $(NL)^{pi}:=(NL/N)^{pi}$): 
\begin{enumerate}

\item $(NL/N)^{sep}=NL^{sep}=N(\th^{p^{n-m_{f,N}}})\simeq N[x]/(f^{sep}_N)$, $[NL^{sep}:N]=\deg(f^{sep}_N(x))=s$ and $f^{sep}_N(x)\in \Irr_m(N[x])$ is the minimal polynomial of the element $\th^{n-m_{f,N}}$ over the field $N$.

\item  $(NL/N)^{pi}=N\bigg( 
\l_0^\frac{1}{p^{m_{f,N}+m_{f,N(\th)}}}, \ldots ,  \l_{s-1}^\frac{1}{p^{m_{f,N}+m_{f,N(\th)}}}\bigg)\supseteq NL^{pi}=N\bigg( 
\l_0^\frac{1}{p^{m_{f,N}}}, \ldots ,  \l_{s-1}^\frac{1}{p^{m_{f,N}}}\bigg)$ and
\begin{eqnarray*}
[(NL)^{pi}:N]&=& p^{m_{f,N(\th)}},\\
 m_{f,N(\th)}&=&\max\Big\{m'=0,1,\ldots , n-m_{f,N} \, | \, \th^{p^{n-m_{f,N}-m'}}\in (NL)^{pi}(NL)^{sep} \Big\}\\
 &=&\max\Big\{m'=0,1,\ldots , n-m_{f,N}  \, | \, (NL)^{p^{n-m_{f,N} -m'}}\subseteq  (NL)^{pi}(NL)^{sep} \Big\}.
\end{eqnarray*}
  In particular, the number $m_{f,N(\th)}$ is an isomorphism invariant of the field extension $NL/N$.

\item $(NL)^{pi}(NL)^{sep}=(NL)^{pi}\t_N (NL)^{sep}=(NL)^{pi}\Big(\th^{p^{n-m_{f,N}-m_{f,N(\th)}}}\Big)\simeq (NL)^{pi}[x]/(f^{sep}_{NL})$, 
\begin{eqnarray*}
[(NL)^{pi}(NL)^{sep}:(NL)^{pi}] &=& s,\;\;  [(NL)^{pi}(NL)^{sep}:(NL)^{sep}] = p^{m_{f,N(\th)}},\\
f^{sep}_{NL} &=& \sum_{i=0}^s\l_i^\frac{1}{p^{m_{f,N}+m_{f, N(\th)}}}x^i\in \Irr_m((NL/N)^{pi}[x])
\end{eqnarray*}
is the minimal polynomial of the element $\th^{p^{n-m_{f,N}-m_{f,N(\th)}}}$ over the field $(NL)^{pi}$. 

\item $NL=(NL)^{pi}\t_N (NL)^{sep}(\th)=(NL)^{pi}\t_N (NL)^{sep}[x]\bigg/\bigg(x^{p^{n-m_{f,N}-m_{f,N(\th)}}}- \th^{p^{n-m_{f,N}-m_{f,N(\th)}}}\bigg)$,  
\begin{eqnarray*}
[NL:(NL)^{pi}\t_N (NL)^{sep}]&=& p^{n-m_{f,N}-m_{f,N(\th)}},\\
 x^{p^{n-m_{f,N}-m_{f,N(\th)}}}- \th^{p^{n-m_{f,N}-m_{f,N(\th)}}}&\in & \Irr_m\Big((NL)^{pi}\t_N (NL)^{sep}[x]\Big)
\end{eqnarray*}
is the minimal polynomial of the element $\th$ over the field $(NL)^{pi}\t_N (NL)^{sep}$. The finite field extension $L/L^{pi}\t_N L^{sep}$ is a simple purely inseparable field extension of exponent $n-m_{f,N}-m_{f,N(\th)}$. 

\item $\Big(NL/(NL)^{pi}\Big)^{sep}=(NL)^{pi}(NL)^{sep}/(NL)^{pi}$.

\end{enumerate}
\end{theorem}

\begin{proof} The theorem is  Theorem \ref{27Apr26} but for the field extension $N(\th)/N$ rather than $K(\th)/K$. It is  obtained in a straightforward manner from Theorem \ref{27Apr26} using Theorem \ref{A6May26}. 
\end{proof}

{\bf Criterion for $(NL)^{pi}=N$ for a simple field extension $L/K$ and a purely inseparable field extension $N/K$.} For a simple field extension $L/K$  and a purely inseparable field extension $N/K$, Corollary  \ref{Na7May26} is  an explicit criterion for $(NL)^{pi}=N$.
 

\begin{corollary}\label{Na7May26}
Suppose that $K$ is a field of prime characteristic $p>0$, $L/K$ is a simple finite  field extension and $L=K(\th)=K[x]/(f(x))$ where $f(x)=f^{sep}(x^{p^{n}})=\sum_{i=0}^s\l_ix^{ip^n}\in \Irr_m(K[x])$, $\l_i\in K$,  $\l_s=1$ and $N/K$ is a purely inseparable field extension. Then the following statements are equivalent: 
\begin{enumerate}

\item $(NL)^{pi}=N$.

\item Either $n=m_{f,N}$ or $n>m_{f,N}$ and  $\l_i^\frac{1}{p^{m_{f,N}+1}}\not\in NL$ for some index $i\in \{ 0,1,\ldots,  s-1\}$. 

\item $m_{f, N(\th)}=0$.

\item $[NL:(NL)^{sep}]=p^{n-m_{f,N}}$.

\end{enumerate}
\end{corollary} 

\begin{proof}  By Theorem \ref{VB-27Apr26} or diagram (\ref{NL=NLLL}),  $(NL)^{pi}=N$ iff  $m_{f, N(\th)}=0$ iff either $n=m_{f,N}$ or $n>m_{f,N}$ and $\l_i^\frac{1}{p^{m_{f,N}+1}}\not\in NL$ for some index  $i\in \{ 0,1,\ldots,  s-1\}$. Clearly, $m_{f, N(\th)}=0$ iff $[NL:(NL)^{sep}]=p^{n-m_{f,N}}$.
\end{proof}

\begin{example}[An example where $(NL)^{pi}=N$]\label{NLpi=N} Let $p=3\; ( > 2)$ be a prime,  $K = \mathbb{F}_p(u, v)$ be the  field  of rational functions in two variables over $\mathbb{F}_p$ and  $L = K(\th)$ where $\th$ is a root of the irreducible polynomial
$$ 
f(x) = x^{2p^n} + ux^{p^n}+ v.
$$
Clearly, $f^{sep}(x)=x^2 + ux + v$ and $n=\deg_{\rm ins}(f)$. Let $N=\mathbb{F}_p(u^\frac{1}{p^m}, v^\frac{1}{p^l})$ for some natural numbers $m$ and $l$ such that $1\leq m \leq l$ and $m < n=m+1$. Then, by Theorem \ref{A6May26} or Example \ref{f=2nml},  
$$
f_N= x^{2p} + u^\frac{1}{p^m}x^p+ v^\frac{1}{p^m}, \;\; f_N^{sep}=x^2 + u^\frac{1}{p^m}x+ v^\frac{1}{p^m}, \;\; M_{f,N}=\mF_p (u^\frac{1}{p^m}, v^\frac{1}{p^m}),
$$ 
$m_{f,N}=m=n-1$ and  $N(\th)=N[x]/(f_N)=\bigoplus_{i=0}^5 N\th^i$. 

Since $m_{f,N}=m=n-1<n$, to prove that the equality $(NL)^{pi}=N$ holds  it suffices to show that $u^\frac{1}{p^{m+1}}\not\in NL$, by Corollary \ref{Na7May26}.(2). Suppose that $u^\frac{1}{p^{m+1}}\in NL$. We seek a contradiction. Then $u^\frac{1}{p^{m+1}}=\sum_{i=0}^5n_i\th^i$ for some elements $n_i\in N$. Hence, 
$$u^\frac{1}{p^m}=\Big(u^\frac{1}{p^{m+1}}\Big)^p=\bigg(\sum_{i=0}^5n_i\th^i\bigg)^p=\sum_{i=0}^5n_i^p\th^{ip}=n_0^p+n_1^p\th^3+n_2^p\th^6+n_3^p\th^9+n_4^p\th^{12}+n_5^p\th^{15}.
$$
Notice that $\th^6=\alpha \th^3+\beta$, where $\alpha:= -u^\frac{1}{p^m}$ and $\beta :=-v^\frac{1}{p^m}$, and
\begin{eqnarray*}
\th^9&=&\th^3\th^6=\th^3(\alpha \th^3+\beta)=\alpha (\alpha \th^3+\beta)+\beta\th^3=(\alpha^2+\beta)\th^3 +\alpha\beta,\\
\th^{12}&=& \th^3\th^9=(\alpha^2+\beta)\th^6 +\alpha\beta\th^3=(\alpha^2+\beta)(\alpha \th^3+\beta) +\alpha\beta\th^3=(\alpha^3+2\alpha\beta)\th^3+(\alpha^2+\beta)\beta,\\
\th^{15}&=&\th^3\th^9=(\alpha^3+2\alpha\beta)\th^6+(\alpha^2+\beta)\beta\th^3=
(\alpha^3+2\alpha\beta)(\alpha \th^3+\beta)+(\alpha^2+\beta)\beta\th^3\\
&=&\Big((\alpha^3+2\alpha\beta)\alpha + (\alpha^2+\beta)\beta\Big)\th^3+(\alpha^3+2\alpha\beta)\beta.
\end{eqnarray*}
Therefore,
\begin{equation}\label{upm=th}
u^\frac{1}{p^m}=n_0^p+n_2^p\beta+n_3^p\alpha\beta+n_4^p(\alpha^2+\beta)\beta+n_5^p(\alpha^3+2\alpha\beta)\beta.
\end{equation}
The field $M=\mF_p(\alpha, \beta)=M_{f,N}$ contains the field $M':=\mF_p(\alpha^p, \beta^p)$. Since all the elements $n_i^p$ belong to the field $M'$ and $M=\bigoplus_{i,j=0}^{p-1}M'\alpha^i\beta^j$, we see that 
$M\backslash M'\ni u^\frac{1}{p^m}=n_0^p\in M'$, a contradiction.
\end{example}

 {\bf Criteria for $NL = (NL)^{pi}(NL)^{sep}$ where $L/K$ is a simple finite field extension and $N/K$ is a purely inseparable field extension.}  
For a simple field extension $L/K$ of prime  characteristic $p>0$ and $N/K$ is a purely inseparable field extension, Theorem \ref{NB27Apr26} presents new explicit criteria for $NL=(NL)^{pi}(NL)^{sep}$. 

\begin{theorem}\label{NB27Apr26}
Suppose that $K$ is a field of prime characteristic $p>0$,  $L/K$ is a simple finite  field extension and $L=K(\th)=K[x]/(f(x))$ where $f(x)=f^{sep}(x^{p^{n}})=\sum_{i=0}^s\l_ix^{ip^n}\in \Irr_m(K[x])$, $\l_i\in K$,  $\l_s=1$ and $N/K$ is a purely inseparable field extension. Then the following statements are equivalent:

\begin{enumerate}

\item $NL = (NL)^{pi}(NL)^{sep}\Big(=(NL)^{pi}\t_N (NL)^{sep}\Big)$.
 
\item  $\l_i^\frac{1}{p^n}\in NL$ for all $i=1, \ldots , s-1$, i.e. $n=m_{f,N}+m_{f, N(\th)}$. 

\item  $(NL)^{pi}=N\Big( \l_0^\frac{1}{p^n}, \ldots , \l_{s-1}^\frac{1}{p^n} \Big)$. 

\end{enumerate}
\end{theorem}

\begin{proof} $(1\Leftrightarrow 2)$ By Theorem \ref{VB-27Apr26}.(4), $[NL:(NL)^{pi}\t_N (NL)^{sep}]=p^{n-m_{f,L}-m_{f,NL}}$ and the result follows.

$(2\Leftrightarrow 3)$ By Theorem \ref{VB-27Apr26}.(2), the equality $n=m_{f,N}+m_{f, N(\th)}$ is equivalent to the equality $(NL)^{pi}=N\Big( \l_0^\frac{1}{p^n}, \ldots , \l_{s-1}^\frac{1}{p^n} \Big)$. 
\end{proof}
\begin{example}[A counterexample where $L \neq L^{pi}L^{sep}$] Let $L$ and $N$ be as in Example \ref{NLpi=N}. Then $L \neq L^{pi}L^{sep}$: Since $m_{f,N(\th)}=0$ and  $ m_{f,N}=n-1$ (see Example \ref{NLpi=N}), we have that 
$$
n\neq n-1= m_{f,N} +m_{f,N(\th)}
$$
 and the result follows from Theorem \ref{NB27Apr26}.(2). 
\end{example}


\section{New and old criteria for $L = L^{pi}L^{sep}$ in finite field extensions}\label{L=LpiLsep}

Each finite field extension $L/K$  is the compositum $L=L_1\cdots L_\nu$ of simple finite field extensions $L_i=K(\th_i)\simeq K[x]/(f_i)$  where $f_i(x)=f_i^{sep}(x^{p^{n_i}})=\sum_{j=0}^{s_i}\l_{ij}x^{jp^{n_i}}\in \Irr_m(K[x])$ is the minimal polynomial of the element $\th_i$ over $K$ and $s_i=\deg (f_i)$. Theorem \ref{A7May26} is a new explicit criterion  for $L = L^{pi}L^{sep}$ which is given in terms of the coefficients $\l_{ij}$ and the numbers $s_i$ and $n_i$. At the beginning of the section we recall  known criteria for $L = L^{pi}L^{sep}$.\\ 

{\bf Known Criteria for $L = L^{pi}L^{sep}$.} 
Let $L/K$ be a finite field extension of a field $K$ of prime characteristic $p > 0$. Let $L^{sep}$ denote the maximal separable subfield of $L$ over $K$ and let $L^{pi}$ denote the maximal purely inseparable subfield of $L$ over $K$.  Recall that $L^{sep} \cap L^{pi} = K$  and $ L^{pi}L^{sep} \cong L^{sep} \otimes_K L^{pi}$. Because of this rigid structure and diagram (\ref{L=LpiLsep}), we have Theorem \ref{C27Apr26} that presents known criteria for $L=L^{pi}L^{sep}$.

\begin{equation}\label{L=LpiLsep}
\begin{tikzpicture}[scale=1.8, every node/.style={font=\normalsize}]

\node (L) at (0,2.6) {$L$};
\node (A) at (0,0) {$K$};
\node (B) at (-1.2,0.9) {$L^{pi}$};
\node (D) at (1.2,0.9) {$L^{sep}$};
\node (C) at (0,1.7) {$L^{pi}L^{sep} \;=\; L^{pi}\!\otimes L^{sep}$};

\draw (L) -- node[right] {pi} (C);

\draw (C) -- node[right] {pi} (D);
\draw (D) -- node[right] {sep} (A);
\draw (A) -- node[left]  {pi} (B);
\draw (B) -- node[left]  {sep} (C);

\end{tikzpicture}
\end{equation}

Theorem \ref{C27Apr26} comprises known criteria for $L = L^{pi}L^{sep}$.

\begin{theorem}\label{C27Apr26}
Suppose that $L/K$ is a finite field extension of prime characteristic $p>0$.  Then the following statements are equivalent:

\begin{enumerate}

\item $L = L^{pi}L^{sep}$.

\item (The Degree Criterion) $[L:K] = [L^{pi}:K] \cdot [L^{sep}:K]$. 
     
\item (Separability over the Purely Inseparable Part) The extension $L/L^{pi}$ is a separable field  extension.
 
\item (Equality of the Inseparable Degree)  The maximal purely inseparable subfield accounts for the entirety of the inseparable degree of the extension:
    $$ [L^{pi}:K] = [L:K]_i $$
    where $[L:K]_i$ denotes the inseparable degree of $L/K$ (which is equal to $[L:L^{sep}]$).
\end{enumerate}
\end{theorem}

\begin{proof} $(1\Leftrightarrow 2)$ Since $L \supseteq  L^{pi}L^{sep}=L^{pi}\t L^{sep}$, the equality $L = L^{pi}\t L^{sep}$ holds iff $[L:K] = [L^{pi}:K] \cdot [L^{sep}:K]$.

$(1\Leftrightarrow 3)$ By diagram (\ref{L=LpiLsep}), $L = L^{pi}\t L^{sep}$  iff  $L/L^{pi}$ is a separable field  extension.

$(3\Leftrightarrow 4)$ The equivalence is obvious.
\end{proof}

{\bf Criteria for $L = L^{pi}L^{sep}$ where $L=L_1\cdots L_\nu$ is the compositum of simple finite field extensions $L_i$.}

\begin{theorem}\label{A7May26}
Suppose that $L/K$ is a finite field extension of prime characteristic $p>0$ which  is the compositum $L=L_1\cdots L_\nu$ of       simple field extensions $L_i=K(\th_i)\simeq K[x]/(f_i)$, $i=1, \ldots , \nu$ where $f_i(x)=f_i^{sep}(x^{p^{n_i}})=\sum_{j=0}^{s_i}\l_{ij}x^{jp^{n_i}}\in \Irr_m(K[x])$ and $\deg (f_i)=s_ip^{n_i}$.  Then the following statements are equivalent:

\begin{enumerate}

\item $L = L^{pi}L^{sep}$.

\item $\l_{ij}^\frac{1}{p^{n_i}}\in L^{pi}$ for $i=1, \ldots , \nu$ and $j=0,1, \ldots , s_i-1$. 

\item $L^{pi}= K\bigg(\l_{ij}^\frac{1}{p^{n_i}}\bigg| i=1, \ldots , \nu; j=0,1, \ldots , s_i-1\bigg)$.
     
\item $L^{pi}\supseteq K\bigg(\l_{ij}^\frac{1}{p^{n_i}}\bigg| i=1, \ldots , \nu; j=0,1, \ldots , s_i-1\bigg)$.

\end{enumerate}
\end{theorem}

\begin{proof} $(1\Leftrightarrow 2)$ By Theorem \ref{C27Apr26}.(3), the equality $L = L^{pi}L^{sep}$ holds iff the field extension $L/L^{pi}$ is a separable field  extension iff for each $i=1, \ldots , \nu$, the polynomial
 $f_{i, L^{pi}}(x)=\sum_{j=0}^{s_i}\l_{ij}^\frac{1}{p^{n_i}}  x^j$ is the minimal polynomial of the element $\th_i$ over the field $L^{pi}$, by Theorem \ref{A6May26}.(1)   (since $L=L^{pi}(\th_1,\ldots , \th_\nu)$). 

$(3\Rightarrow 2\Leftrightarrow 4)$ Clear. 

$(2\Rightarrow 3)$ Let $N_i:=K\bigg(\l_{i,0}^\frac{1}{p^{n_i}}, \ldots, \l_{i,s_i-1}^\frac{1}{p^{n_i}}  \bigg)$. Suppose that statement 2 holds. Then, by Theorem \ref{B27Apr26}, $L_i=L_i^{pi}L_i^{sep}$ where   $L_i^{pi}=N_i$ 
 for $i=1, \ldots , \nu$. Therefore, 
\begin{eqnarray*}
 L&=&\prod_{i=1}^\nu L_i=\prod_{i=1}^\nu L_i^{pi}L_i^{sep}=\prod_{i=1}^\nu L_i^{pi} \prod_{i=1}^\nu L_i^{sep}\stackrel{{\rm Thm.}\, \ref{A27Apr26}}{=}\prod_{i=1}^\nu L_i^{pi}\t  \prod_{i=1}^\nu L_i^{sep}
\end{eqnarray*}
since $\prod_{i=1}^\nu L^{pi}_i\subseteq L^{pi}$ and  $\prod_{i=1}^\nu  L_i^{sep} \subseteq L^{sep}$ and $L^{pi}L^{sep}=L^{pi}\t L^{sep}$ (Theorem \ref{A27Apr26}). Now, two inclusions above and  the inclusions
$$
L^{pi}\t L^{sep}=L^{pi}L^{sep}\subseteq L=
\prod_{i=1}^\nu L_i^{pi}\t \prod_{i=1}^\nu  L_i^{sep} \subseteq L^{pi}\t L^{sep}
\subseteq L$$
imply that $L=L^{pi} L^{sep}=L^{pi}\t L^{sep}$,  
$$
L^{pi}=\prod_{i=1}^\nu L_i^{pi}=\prod_{i=1}^\nu N_i=K\bigg(\l_{ij}^\frac{1}{p^{n_i}}\bigg| i=1, \ldots , \nu; j=0,1, \ldots , s_i-1\bigg)
$$ 
and $L^{sep}= \prod_{i=1}^\nu L_i^{sep}$. 
\end{proof}




 {\bf Licence.} For the purpose of open access, the author has applied a Creative Commons Attribution (CC BY) licence to any Author Accepted Manuscript version arising from this submission.

{\bf Declaration of interests.} The authors declare that they have no known competing financial interests or personal relationships that could have appeared to influence the work reported in this paper.

{\bf Disclosure statement.} No potential conflict of interest was reported by the author.

{\bf Data availability statement.} Data sharing not applicable – no new data generated.\\

\small{
}

School of Mathematical  and Physical Sciences

Division of Mathematics

University of Sheffield

Hicks Building

Sheffield S3 7RH

UK

email: v.bavula@sheffield.ac.uk

\end{document}